\RequirePackage{fix-cm}
\documentclass[smallextended]{svjour3}       
\smartqed  
\usepackage{graphicx}
\usepackage{lineno}
\usepackage{amsmath,amssymb}
\usepackage{algorithmic}
\usepackage{enumerate}
\usepackage{multicol,multirow}
\usepackage{graphicx}
\usepackage{mathrsfs}
\usepackage{epsfig}
\usepackage{subfigure}
\usepackage{ntheorem}
\usepackage{color}
\usepackage{cite}
\usepackage[justification=centering]{caption}
\usepackage{algorithm}
\numberwithin{equation}{section}
\numberwithin{table}{section}

\usepackage{epstopdf}
\usepackage[pdftex,linkcolor=blue,citecolor=blue]{hyperref}
 \let\oldequation\equation
 \let\oldendequation\endequation
 \renewenvironment{equation}{\linenomathNonumbers\oldequation}{\oldendequation\endlinenomath}

\numberwithin{equation}{section}
\numberwithin{example}{section}
\numberwithin{remark}{section}
\numberwithin{lemma}{section}
\numberwithin{definition}{section}
\numberwithin{algorithm}{section}
\numberwithin{theorem}{section}
\numberwithin{figure}{section}
\numberwithin{table}{section}
\numberwithin{corollary}{section}

\newcommand{\be}{\begin{equation}}
\newcommand{\ee}{\end{equation}}

\journalname{xxx}
\begin{document}

\title{Fast adaptive tubal rank-revealing algorithm for  t-product based tensor approximation
}


\author{Qiaohua Liu \and  Jiehui Gu
}


\institute{ Qiaohua Liu,  Jiehui Gu  \at Department of Mathematics, Shanghai University and Newtouch Center for Mathematics,  Shanghai 200444, People's Republic of China. Q. Liu is   supported by the National Natural Science Foundation of China under grant 12471355; the Natural Science Foundation of Shanghai under grant 23ZR1422400. \emph{Email address:} qhliu@shu.edu.cn(Q. Liu). }

\date{Received: date / Accepted: date}

\maketitle

\begin{abstract} Color images and video sequences can be modeled as three-way tensors, which admit low tubal-rank approximations via convex surrogate minimization. This optimization problem is efficiently addressed by tensor singular value thresholding (t-SVT). To mitigate the computational burden of tensor singular value decomposition (t-SVD) in each iteration, this paper introduces an adaptive randomized algorithm for tubal rank revelation in data tensors \(\mathcal{A}\). Our method selectively captures the principal information from frontal slices in the Fourier domain using a predefined threshold, obviating the need for   priori tubal-rank and Fourier-domain singular values estimations  while providing an explicit tensor approximation. Leveraging optimality results from matrix randomized SVD, we establish theoretical guarantees demonstrating that the proposed algorithm computes low tubal-rank approximations within constants dependent on data dimensions and the Fourier-domain singular value gap. Empirical evaluations validate its efficacy in image processing and background modeling tasks.\\
\end{abstract}


\section{ { {Introduction}}} 
\label{intro-sec}

The problem of exploiting low-dimensional structure in high-dimensional data is taking on increasing importance in diverse fields of applications, including data analysis, image/video and signal processing,   and so on.  These   can be regarded as matrix approximation   or tensor approximation problems.
Traditional matrix-based data analysis using  singular value decomposition (SVD) \cite{gv2} is inherently two-dimensional. It provides an optimal rank-$k$ approximation in the matrix spectral norm and Frobenius norm, but  limits its usefulness in extracting information from a multidimensional perspective.

 The models that based on tensor multilinear data analysis have shown to be  capable of taking full advantage of the multilinear structures and providing better understanding and more precision.
The multilinear data analysis lies tensor decomposition, which
commonly takes three forms: CANDECOMP/PARAFAC (CP) decomposition \cite{hi}, Tucker decomposition \cite{tu}, tensor-SVD \cite{km}, tensor train decomposition \cite{cw3} and so on. They can be regarded as the generalization of matrix SVD to tensors.  These decompositions are now widely used in other application areas
such as computer vision \cite{gq,tm}, web data mining \cite{fss,szl} and so on.

However, not all of these tensor decompositions are   convenient   to handle the low-dimensional approximation problem.
 The CP rank, defined as the smallest number of rank one tensor decomposition, is generally NP-hard to compute.
The Tucker rank is a vector defined  by the rank of the unfolded matrices  of a tensor based on
matricization method, especially for a three-order tensor, the Tucker3 decomposition computes a core tensor and  factor matrices with orthonormal columns\raisebox{0.06mm}{------}a higher order SVD (HOSVD) \cite{lmv}.
  One can  specify the rank of the factor matrices and thus obtain a so-called best rank-($R_1, R_2, R_3$) Tucker3 decomposition.
However, unlike the matrix case, this best approximation cannot be obtained by truncating the full HOSVD. Besides, unfolding a $k$-way tensor
using matricization methods would generate $k$ matrices with the same number of entries as the tensor. Each unfolding might
destroy the original multi-way structure of the data and result in  vital information loss, and because of
the large number of entries of a tensor,  the computation would be inefficient  when recovering  each unfolded matrix and combining them into a tensor at each iteration.

The  tensor-SVD (t-SVD) introduced by Kilmer et al.  \cite{km} for 3-way tensors is based on the concept of tensor-tensor product (t-product) with a suitable algebraic structure.
 The advantage of such an approach over CP and Tucker decompositions is that its computation can be
 parallelizable and
  the resulting algebra and corresponding analysis enjoys
 many similar properties to matrix algebra and analysis. The tubal rank is defined via t-SVD and  the truncated t-SVD does give a compressed result that is theoretically optimal in the Frobenius and spectral norms  \cite{cw,km}. However,
for large-scale 3-way tensors,  the computation of a low tubal rank tensor approximation via full t-SVD might  suffer too much computational cost and storage requirement.
  In this paper  we propose an efficient randomized low tubal rank-revealing algorithm for the third-order high-dimensional  tensor, and provides its  low-tubal rank  approximation of the tensor with acceptable accuracy.

By trading  the optimality property of the matrix SVD with computational efficiency, the popular randomized schemes \cite{hmt,llj,lwm,ma,mrt,rlb,rxb,wlr,woo}   have shown their  effectiveness in low-rank matrix approximation problems.   Zhang et al. \cite{zsk} proposed randomized t-SVD, by implanting
  the idea of matrix randomized SVD \cite{hmt} to the t-product-based tensor approximation problems.   There are also many works that generate  tensor low-rank approximations that have similar rank-revealing properties to the randomized t-SVD, such as randomized tensor CX and CUR decompositions \cite{tm}, randomized tensor URV \cite{cw2}.

However,   these randomized rank-revealing methods for computing tensor low-rank representation always specify an estimation of the tubal rank beforehand.
Che et al. \cite{cw} proposed a randomized tensor singular value thresholding method \cite{cw}, by utilizing a threshold $\tau$
and a prior numerical tubal rank estimation  $k$ to compute a  low-rank approximation of  t-SVD. Song et al. \cite{sx} proposed low tubal rank tensor completion based on
singular value factors. Both methods first used randomized  sampling to
project  the high-dimensional $n_1\times n_2\times n_3$ tensor onto lower dimensional $n_1\times k\times n_3$ ($k\ll n_2$) tensor subspace with the estimated tubal rank, and then detected a compressed t-SVD  of the lower dimensional tensor  based on the given threshold  or a rank-decrease method based on spectral energy ratio or spectral gap \cite{lmw,wyz,zll}.
 If $k$ is underestimated, one might get a bad tensor approximation, and an overestimate of $k$ will increase the computational cost.
Recently  Ahmadi-Adl \cite{ad} proposed a randomized fixed-precision rank-revealing algorithm for approximating the t-SVD, by generalizing the work in \cite{mv,ygl} from matrices to tensors. The algorithm detects the Fourier-domain tensor range blocks by blocks and finds an approximation of the tensor, until the residual energy is less than the given precision in Frobenius norm. However the algorithm is restricted in the applications of  some convex tensor optimization problem in that the determination of the tubal rank depends on the
prior estimation of the fixed-precision.
 For example, in data denoising and completion problems,
   one needs to recover the low tubal rank  components  using tensor robust principal component analysis (\cite{lcl,lcl2}) and tensor completion (\cite{htz,sn}) algorithms. These algorithms require
     to derive low tubal rank approximation of 3-way tensor with a given singular value threshold in the iterations.

     The focus of this study lies in the low tubal-rank approximation of given tensors with respect to fixed threshold. The contribution made by this work can be summarized into three aspects.

     \begin{itemize}
\item For the randomized SVD of an $m\times n$  matrix $A$,  given a specified  target rank $k$, the computed orthogonal rank-$i$ ($1\le i\le k$)  approximate basis $\tilde  U_i$ of the range ${\rm Ran}(A)$ endows the projection  $\tilde U_i\tilde U_i^TA$ with a remarkable optimality property. 
              Inspired by this insight, we propose a randomized tubal rank-revealing algorithm for approximate tensors. Operating under a pre-specified tubal rank and fixed threshold, the algorithm computes the tensor's multirank and constructs an approximate tensor. We also provide theoretical analysis for the approximate tensor's properties and performance guarantees.

\item  Inspired by the adaptive block randomized rank-revealing algorithm \cite{liu} for matrices, we dispense with the pre-specified tubal rank guess and   propose an adaptive tubal rank-revealing algorithm for tensor approximation of \(\mathcal{A}\). We provide probabilistic error bounds for the error of the approximate tensor. The analysis of the results is not a simple conclusion drawn from \cite{liu}; instead, it breaks the computed rank assumption in \cite{liu}. The derived error results show that the proposed algorithm can compute a low tubal-rank approximation, with the approximation accuracy up to certain constants that depend on the data dimension and the Fourier-domain singular value gap from   optimal. If the Fourier-domain singular values exhibit a fast decay rate, the corresponding results will yield tighter error bounds.

\item The proposed algorithm does not produce an exact  tensor singular value thresholding (t-SVT) algorithm. But it can still give approximate solutions of tensor RPCA problem arising in color image compression and video processing problems. Experimental results show that the proposed algorithm is effective and behaves more efficient than the truncated t-SVD algorithms.

  \end{itemize}

Throughout this paper, we use the following notations. Given an $m\times n$  matrix $A$,    $A^T (A^H)$, Ran$(A)$ and $\sigma_j(A)$ denote the transpose (conjugate transpose), the range, and the   $j$-th largest singular value of $A$, respectively. Let the function $Q={\sf orth}(A)$ denote the orthonormalization of the columns of $A$, which can be achieved by
 calling a packaged QR factorization, say in Matlab, we call [$Q$, $\sim$ ]={\sf qr}($A$,0). If $Q$ and $Y$ have orthonormal columns,
 $Z={\sf orth2}(Y-Q(Q^TY))$ means twice-orthogonalization of  $Y$ against $Q$, i.e.,
\[
Y:=Y-Q(Q^TY),\quad Y:={\sf orth}(Y),\quad Z= Y-Q(Q^T Y).
\]
 For a Hermitian matrix $M$, the function $[V,D]={\sf eig_\downarrow}(M)$ means we compute the eigenpair of $M$ with the eigenvalues in $D$  arranged in descending order, and the columns of $V$ being corresponding  eigenvectors.

\section{Preliminaries}

In this section, we review some basic definitions, operations and well-known results of tensors.

\subsection{Basic definitions}

A tensor is a multidimensional array, and the order of the tensor
is the number of dimensions of this array. In this paper, we focus on the third-order tensor ${\cal A}\in {\mathbb R}^{n_1\times n_2\times n_3}$, whose entry is denoted by MATLAB indexing notation, i.e., ${\mathcal A}_{ijk}$.  The notations ${\cal A}(:,j,k)$, ${\cal A}(i,:,k)$ and ${\cal A}(i,j,:)$
are the ($j,k$)-th column fiber, the ($i,k$)-th row fiber and the ($i,j$)-th
tube fiber  of ${\cal A}$; ${\cal A}(i,:,:)$, ${\cal A}(:,j,:)$, ${\cal A}(:,:,k)$ are the $i$-th horizontal slice, $j$-th lateral slice and $k$-th frontal slice. For simplicity,
we denote ${\cal A}_{(i)}={\cal A}(i,:,:)$,  $\vec{{\cal A}}_j={\cal A}(:,j,:)$ and   ${\cal A}^{(k)}={\cal A}(:,:,k)$.

For a given real tensor ${\cal A}$, by using the MATLAB commands {\sf fft}  and {\sf ifft}, we denote by $\widehat {\cal A}$ the result of applying the Discrete Fourier Transform (DFT) on ${\cal A}$ along the third dimension, i.e., $\widehat {\cal A}={\sf fft}({\cal A},[~],3)$.  Analogously, one can also compute ${\cal A}$ from $\widehat {\cal A}$ via the inverse DFT, using
${\cal A}={\sf ifft}(\widehat {\cal A},[~],3)$. For simplicity, we denote $\widehat {\cal A}={\sf fft}_3({\cal A})$, and ${\cal A}={\sf ifft}_3(\widehat {\cal A})$.
We also  denote by $\widehat { A}\in {\mathbb C}^{n_1n_3\times n_2n_3}$ the block diagonal matrix corresponding to the frontal slice $\widehat{\cal A}^{(k)}$. That is,
\[
\widehat {\bf A}={\sf bdiag}(\widehat {\cal A}^{(1)},\ldots, \widehat {\cal A}^{(n_3)})=\left[\begin{array}{cccc}
\widehat {\cal A}^{(1)}&0_{n_1\times n_2}&\ldots& 0_{n_1\times n_2}\\
0_{n_1\times n_2}&\widehat {\cal A}^{(2)}&\ldots&0_{n_1\times n_2}\\
\vdots&\vdots&\ddots&\vdots\\
0_{n_1\times n_2}&0_{n_1\times n_2}&\ldots&\widehat {\cal A}^{(n_3)}
\end{array}\right].
\]

\begin{definition} [\cite{km}] (tensor product) Given two tensors ${\cal A}\in {\mathbb R}^{n_1\times n_2\times n_3}$ and
${\cal B}\in {\mathbb R}^{n_2\times n_4\times n_3}$, the tensor product (t-product) ${\cal A}*{\cal B}$ is the $n_1\times n_4\times n_3$ tensor, which is defined
as
\[{\cal A}*{\cal B}={\sf fold}({\rm circ}({\cal A})\cdot {\sf unfold}({\cal B})),\]
 where ${\sf fold}({\sf unfold}({\cal B}))={\cal B}$,
\[
{\sf circ}({\cal A})=\left[\begin{array}{cccc}
{\cal A}^{(1)}&{\cal A}^{(n_3)}&\ldots&{\cal A}^{(2)}\\
{\cal A}^{(2)}&{\cal A}^{(1)}&\ldots&{\cal A}^{(3)}\\
\vdots&\vdots&\ddots&\vdots\\
{\cal A}^{(n_3)}&{\cal A}^{(n_3-1)}&\ldots&{\cal A}^{(1)}
\end{array}\right],
\]
and
\[
{\sf unfold}({\cal B})=\left[{\cal B}^{(1)^T}~~ {\cal B}^{(2)^T}~ \ldots~ {\cal B}^{(n_3)^T}\right]^T.
\]

\end{definition}

The t-product is associative \cite{km}, i.e., ${\cal A}*({\cal B}*{\cal C})=({\cal A}*{\cal B})*{\cal C}$. 
 Moreover, the block circulant matrix ${\rm circ}({\cal A})$ can be diagonalized into block diagonal matrix $\widehat {\bf A}$ by the $n_3\times n_3$ unitary DFT matrix $ F_{n_3}$ as \cite{km}
\begin{equation}
\widehat  {\bf A}=(F_{n_3}\otimes I_{n_1}){\rm circ}({\cal A})(F_{n_3}^H\otimes I_{n_2}),\label{00}
\end{equation}
where  the unitary matrix  $F_{n_3}$ has the $(i,j)$-entry as $(F_{n_3})_{ij}=\omega^{(i-1)(j-1)}/\sqrt{n_3}$ with $\omega=\exp(-2\pi\sqrt{-1})$.
According to \cite{rr}, the block diagonal matrices in $\widehat  {\bf A}$ has some conjugate symmetry property that $\widehat {\cal A}(:,:,1)$ is real and  ${\rm conj}(\widehat {\cal A}(:,:,i))=\widehat {\cal A}(:,:,n_3-i+2)$ for $i=2,\ldots, n_3$. Conversely, if this conjugate symmetry property of complex $\widehat {\cal A}$ is preserved, then
the operation ${\cal A}={\sf ifft}_3({\cal \widehat A})$ guarantees that ${\cal A}$ is a real tensor.

\begin{definition}[\cite{km}] (conjugate transpose) The conjugate transpose of a tensor ${\cal A}\in {\mathbb R}^{n_1\times n_2\times n_3}$ is an $n_2\times n_1\times n_3$ tensor
${\cal A}^T$ obtained by conjugate transposing each of the frontal slices and then reversing the order of transposed frontal slices 2 through $n_3$.
\end{definition}


%

\begin{definition}[\cite{km}] (identity tensor and f-diagonal tensor) The identity tensor ${\cal I}_{nnn_3}\in {\mathbb R}^{n\times n\times n_3}$ is the tensor whose first frontal slice is the $n\times n$
identity matrix, and whose other frontal slices are all zeros.  A tensor is called  $f$-diagonal if each of its frontal slices is a diagonal matrix.
\end{definition}


\begin{definition}[\cite{km}] (orthogonal tensor) A tensor ${\cal Q}\in {\mathbb R}^{n\times n\times n_3}$ is orthogonal if it satisfies ${\cal Q}^T*{\cal Q}={\cal Q}*{\cal Q}^T={\cal I}_{nnn_3}$.
If  for $k\le n$ and ${\cal Q}\in {\mathbb R}^{n\times k\times n_3}$ satisfies ${\cal Q}^T*{\cal Q}={\cal I}_{kkn_3}$,
then ${\cal Q}$ is said to be partially orthogonal.
\end{definition}

\begin{definition}[\cite{km}]\label{def:2.2} (tensor spectral and Frobenius norms) The tensor spectral and Frobenius norms of ${\cal A}\in {\mathbb R}^{n_1\times n_2\times n_3}$, denoted as
$\|{\cal A}\|_{\xi}$ with $\xi=2, F$, are defined as
\[
\begin{array}{l}
\|{\cal A}\|_2:=\|\widehat  {\bf A}\|_2,\qquad
\|{\cal A}\|_F:=(\sum\limits_{i=1}^{n_3}\|{\cal A}^{(i)}\|_F^2)^{1/2}={1\over \sqrt{n_3}}\|\widehat {\bf A}\|_F.
\end{array}
\]
\end{definition}

\begin{definition}[\cite{km}] (tensor multi-rank and  tubal rank) The tensor
multi-rank  ${\rm rank}_m({\cal A})$ of ${\cal A}$ is the vector ${\rm rank}_m({\cal A})=[r_1, r_2,\ldots, r_{n_3}]^T\in {\mathbb R}^{n_3}$  with $r_i={\rm rank}(
\widehat {\cal A}^{(i)})$.
The tensor tubal rank, denoted by $\nu={\rm rank}_t({\cal A})$, is defined as the maximum entries of ${\rm rank}_m({\cal A})$, i.e., ${\rm rank}_t({\cal A})=\max_i\{r_i\}$.
\end{definition}

\subsection{Review of t-SVD and randomized SVD}


\begin{theorem}(t-SVD and and truncated t-SVD\cite{km}) Let ${\cal A}\in {\mathbb R}^{n_1\times n_2\times n_3}$. Then it can be factored as
\[
{\cal A}={\cal U}*{\cal S}*{\cal V}^T=\sum\limits_{i=1}^{\min\{n_1,n_2\}}\vec{\cal U}_i*{\bf s}_{i}^{\rm tub}*\vec{\cal V}_{i}^T,\quad {\bf s}_i^{\rm tub}:={\cal S}(i,i,:),
\]
where ${\cal U}\in {\mathbb R}^{n_1\times n_1\times n_3}$, ${\cal V}\in {\mathbb R}^{n_2\times n_2\times n_3}$ are orthogonal and 
the $f$-diagonal tensor ${\cal S}\in {\mathbb R}^{n_1\times n_2\times n_3}$ is called the singular value tensor, ${\bf s}_i^{\rm tub}$ is the $i$-th singular value tube fiber. If $k$ is a target truncation term, then the truncated t-SVD ({\sf t-tsvd}) of ${\cal A}$ is defined as
\[
{\cal A}_{k}={\cal U}_k*{\cal S}_k*{\cal V}_k^T,
\]
where ${\cal U}_k={\cal U}(:,1:k,:)\in {\mathbb R}^{n_1\times k\times n_3}$, ${\cal V}_k={\cal V}(:,1:k,:)\in {\mathbb R}^{n_2\times k\times n_3}$ are partially orthogonal, and ${\cal S}_k={\cal S}(1:k,1:k,:)\in{\mathbb R}^{k\times k\times n_3}$ is an f-diagonal tensor.
\end{theorem}


\begin{theorem}[\cite{cw,km}]\label{thm:2.1} Suppose that ${\cal A}_k$ is the {\sf t-tsvd} of ${\cal A}\in {\mathbb R}^{n_1\times n_2\times n_3}$ with a given target truncated term $k$. Let ${\mathbb S}=\{{\cal C}={\cal X}*{\cal Y}: {\cal X}\in {\mathbb R}^{n_1\times k\times n_3}, {\cal Y}\in {\mathbb R}^{k\times n_2\times n_3}\}$,
and $\widehat\sigma_j^{(i)}=\widehat {\cal S}(j,j,i)$ is the $j$-th singular value corresponding to the $i$-th frontal slice in the Fourier domain. Then
 ${\cal A}_{k}={\rm argmin}\|{\cal A}-\tilde {\cal A}\|_{\xi}$ for  all $\tilde {\cal A}\in {\mathbb S}$ and $\xi=2, F$.
  Therefore, $\|{\cal A}-{\cal A}_{k}\|_\xi$ is the theoretical minimal error, given by
\[
\|{\cal A}-{\cal A}_{k}\|_\xi=\left\{
\begin{array}{ll}
\max\limits_{i=1,2,\ldots, n_3}\widehat \sigma_{k+1}^{(i)},&\xi=2,\\
\left({1\over n_3}\sum\limits_{i=1}^{n_3}\sum\limits_{j=k+1}^{n_3}(\widehat \sigma_j^{(i)})^2\right)^{1/2},&\xi=F.
\end{array}\right.
\]
%
\end{theorem}

 Algorithm \ref{alg:2.1} computes a $k$-term truncated t-SVD  of an  ${n_1\times n_2\times n_3}$ tensor, and it costs  ${\cal O}(n_1n_2n_3k)$ flops. In the Fourier domain, the truncated t-SVD enjoys the advantage of parallel computation of the SVD of $\widehat {\cal A}^{(i)}$ in an independent manner.

\begin{algorithm}
\caption{$k$-term truncated t-SVD ({\sf t-tsvd})\cite{lcl}}\label{alg:2.1}

\begin{algorithmic}[1]
\REQUIRE A tensor ${\cal A}\in {\mathbb R}^{n_1\times n_2\times n_3}$ and a target truncated term $k$.

\ENSURE ${\cal U}_k\in {\mathbb R}^{n_1\times k\times n_3}$, ${\cal V}_k\in {\mathbb R}^{n_2\times k\times n_3}$, and ${\cal S}_k\in {\mathbb R}^{k\times k\times n_3}$.

\STATE Compute $\widehat {\cal A}={\sf fft}_3({\cal A})$.

\FOR  {$i=1,2,\ldots, \lceil {n_3+1\over 2}\rceil$}

\STATE Compute the SVD of $\widehat {\cal A}(:,:,i)$ as $\widehat {\cal A}(:,:,i)=USV^H$.

\STATE Form $U_k, V_k$ and $S_k$ by truncating $U, V$ and $S$ with the target  term $k$.

\STATE Set $\widehat {\cal U}_k(:,:,i)=U_k, \widehat {\cal V}_k(:,:,i)=V_k$ and $\widehat {\cal S}_k(:,:,i)=S_k$.

\ENDFOR

\FOR {$i=\lceil {n_3+1\over 2}\rceil+1,\ldots, n_3$}

\STATE Set $\widehat {\cal U}_k(:,:,i)={\rm conj}(\widehat {\cal U}_k(:,:,n_3-i+2))$, $\widehat {\cal V}_k(:,:,i)={\rm conj}(\widehat {\cal V}_k(:,:,n_3-i+2))$,

\qquad~~  and $\widehat S_k(:,:,i)=\widehat S_k(:,:,n_3-i+2)$.

\ENDFOR

\RETURN ${\cal U}_k={\sf ifft}_3(\widehat {\cal U}_k)$, ${\cal V}_k={\sf ifft}_3(\widehat {\cal V}_k)$ and ${\cal S}_k={\sf ifft}_3(\widehat S_k)$.
\end{algorithmic}
\end{algorithm}

\begin{algorithm}
\caption{(RSVD,\cite{hmt})  Basic randomized SVD algorithm with power scheme}  \label{rQB}
\begin{algorithmic}[1]
\REQUIRE $m\times n$ matrix $M$ with $m\ge n$, integer $k>0$, oversampling parameter $p$ and $\ell=k+p\ll n$ and power parameter $q$.

\ENSURE  A rank-$k$ approximation $\tilde M_k$.

\STATE Draw a random $n\times \ell$  Gaussian matrix $G$ via $G={\sf randn}(n,\ell)$.

\STATE Compute an orthogonal column basis  for $M\Omega$, we denote $Q={\sf orth}((MM^T)^qMG)$.

\STATE  Compute $B=Q^TM$.

\STATE Compute the SVD: $B=\tilde Z\tilde \Sigma\tilde  V^T$.

\STATE  Compute $\tilde U=Q\tilde Z$ and get $\tilde M_k=\tilde U_{k}\tilde\Sigma_{k}\tilde V_{k}^T$ from truncated SVD of $QB$.
\end{algorithmic}
\end{algorithm}

The randomized t-SVD algorithm proposed by Zhang et al. \cite{zsk}    is an extension of matrix randomized SVD (RSVD) \cite{hmt}.
With strong theoretical guarantee, the idea of RSVD is to trade accuracy with efficiency, by using randomization and sampling strategy to convert the low-rank approximation of large input matrix into the low rank approximation problem of a small matrix. The algorithm is described in Algorithm \ref{rQB}, where
the oversampling parameter $p$ and power parameter $q$  are used to enhance the quality of $Q$. In practice, $q\le 2$ is enough to achieve a good accuracy.
Let $\tilde U_i$  contain the first $i$ columns of $\tilde U$. The theorem below shows that $\tilde U_i\tilde U_i^TM$ gives the optimal rank-$i$ approximation to $M$ within the range of $\tilde U$.

\begin{theorem}\label{thm:2.2} Let $1\le k\le {\rm rank}(M)$. For the $m\times n$ real matrix $M$, consider its SVD:
\[
M=U\Sigma V^T=[U_k\quad U_{k,\bot}]\left[\begin{array}{cc}\Sigma_k&\\&\Sigma_{k,\bot}\end{array}\right]\left[\begin{array}c V_{k}^T\\ V_{k,\bot}^T\end{array}\right],
\]
where $\Sigma_k\in {\mathbb R}^{k\times k}$, $U_k, V_k$ contain the first $k$ columns of $U, V$, respectively.
Let $M_k=U_k\Sigma_kV_k^T$ be the best rank-$k$ approximation to $M$.
Assume the starting Gaussian matrix $G\in {\mathbb R}^{n\times \ell}$ with the oversampling columns $p=\ell-k\ge 2$, and   the notations $Q, \tilde Z$ and $\tilde U=Q\tilde Z$ are obtained from Algorithm \ref{rQB}.
If $\tilde U_i$ denotes  the first $i$($1\le i\le k$) columns of  $\tilde U$, then
$Q(Q^TM)_i$ is the best rank-$i$ approximation to $M$ within the range of $Q$ in the Frobenius norm, i.e.,
\begin{equation}
\|M-Q(Q^TM)_i\|_F^2=\min\limits_{{\rm rank}(Y)\le i} \|M-QY\|_F^2.\label{min}
\end{equation}
In addition,  $Q(Q^TM)_i=\tilde U_i\tilde U_i^TM$.   If $\gamma_i^{4q+2}\sum\limits_{j>k}\big({\sigma_j/ \sigma_{k+1}}\big)^2<{(p-1)/k}$, 
 then for $i=1,2,\ldots,k$,
\begin{eqnarray}
&{}&{\mathbb E}\|M-\tilde U_i\tilde U_i^TM\|_F^2\le  \big(\sum_{j>i}\sigma_j^2\big)+{k\over p-1}\gamma_i^{4q}\big(\sum\limits_{j>k}\sigma_j^2\big),\label{bnd1}\\
&{}&{\mathbb E}\|M-\tilde U_i\tilde U_i^TM\|_2^2\le \sigma_{i+1}^2+{k\over p-1}\gamma_i^{4q}\big(\sum\limits_{j>k}\sigma_j^2\big),\label{bnd2}\\
&{}&\sigma_i^2\Big[1-{k\over p-1}\gamma_i^{4q+2}\big(\sum\limits_{j>k}\big({\sigma_j\over \sigma_{k+1}}\big)^2\Big]\le {\mathbb E}\|\tilde u_i^TM\|_2^2\le \sigma_i^2,\label{bnd3}
\end{eqnarray}
where   $\gamma_i={\sigma_{k+1}/\sigma_i}$ with $\sigma_j$ being the $j$-th largest singular value of $M$,   $\tilde u_i$ is the $i$th column of $\tilde U_i$,
and ${\mathbb E}(\cdot)$ denotes the expectation of an event.
\end{theorem}

\begin{proof} Let $B=Q^TM$ and  $M_{k,\bot}=M-M_k$. The minimizer $B_i=(Q^TM)_i$ of \eqref{min} satisfies the relation
\begin{equation}
\|M-Q(Q^TM)_i\|_F^2\le  \|(I-QQ^H)M_i\|_F^2+\|M_{i,\bot}\|_F^2,\label{eq0}
\end{equation}
which is  proved by \cite[Theorem 3.5]{gu}, where $QB_i=Q(Q^TM)_i=\tilde U_i\tilde U_i^TM$. In terms of the SVD: 
$B=\tilde Z\tilde S\tilde V^T$, we know that the optimal rank-$i$ approximation $B_i$ takes the form $B_i=\tilde Z_i\tilde S_i\tilde V_i^H$, and
\[
\begin{array}{rl}
QB_i=QB_iB_i^\dag B_i=Q\tilde Z_i\tilde Z_i^TB_i=Q\tilde Z_i\tilde Z_i^TB=\tilde U_i\tilde U_i^TM.
\end{array}
\]
Additionally,  by partitioning $V^TG$ as $V^TG=\Big[{G_k\atop G_{n-k}}\Big]$
with $G_k\in {\mathbb R}^{k\times \ell}$ and $G_{n-k}\in {\mathbb R}^{(n-k)\times \ell}$,
then $V^TG$ is still a Gaussian matrix, and the term $\|(I-QQ^T)M_i\|_F$ in \eqref{eq0} can be bounded as
\begin{align}
\|(I-QQ^T)M_i\|_F\le \gamma_i^{2q}\|\Sigma_{k,\bot}G_{n-k}G_k^\dag\|_F=:\Delta.\label{Delta}
\end{align}
This can be proved by following a similar argument in proving \cite[Theorem 7]{sa}.
 Thus, we have derived that
 \begin{equation}
 \|M-Q(Q^TM)_i\|_F^2\le \|M_{i,\bot}\|_F^2+\Delta^2,\label{eq1}
 \end{equation}
 where  by \cite[Prop. 10.1, 10.2]{hmt}, the expectation 
$$
{\mathbb E}\Delta^2={k\over p-1}\gamma_i^{4q}\|\Sigma_{k,\bot}\|_F^2={k\sigma_i^2\gamma_i^{4q+2}\over p-1}\sum_{j>k}({\sigma_j\over \sigma_{k+1}})^2.
$$
 The assertion \eqref{bnd1} comes true. By applying the reverse Eckart and Young theorem \cite[Theorem 3.4]{gu}, we get
$\|M-\tilde U_i\tilde U_i^TM\|_2^2\le \|M_{i,\bot}\|_2^2+\Delta^2$, thus giving \eqref{bnd2}.

For the estimate in \eqref{bnd3}, let $M_{(i)}^\bot=M-\tilde U_i\tilde U_i^TM$ with $M_{(0)}^\bot=M$.   Decompose $M_{(i-1)}^\bot=M_{(i)}^\bot+\tilde u_i\tilde u_i^TM$,  in which     $\tilde u_i\tilde u_i^TM_{(i)}^\bot=0$.
By the formula $\|N\|_F^2={\rm trace }(N^TN)$ we get $\|M_{(i-1)}^\bot\|_F^2=\|M_{(i)}^\bot\|_F^2+\|\tilde u_i\tilde u_i^TM\|_F^2$. Combing this with \eqref{eq1} and the optimality of the rank-$(i-1)$ approximation $M_{i-1}$ of $M$, we get
  \[
\|M_{i-1,\bot}\|_F^2-\|\tilde  u_i\tilde u_i^TM\|_F^2\le\|M_{(i-1)}^\bot\|_F^2-\|\tilde  u_i\tilde u_i^TM\|_F^2=\|M_{(i)}^\bot\|_F^2\le \|M_{i,\bot}\|_F^2+\Delta^2,
  \]
and hence
\[
\|\tilde u_i^TM\|_2^2=\|\tilde  u_i\tilde u_i^TM\|_F^2\ge \|M_{i-1,\bot}\|_F^2-\|M_{i,\bot}\|_F^2-\Delta^2=\sigma_i^2-\Delta^2,
\]
yielding $\sigma_i^2\le \|\tilde u_i^TM\|_2^2+\Delta^2$,
which is the left-side inequality of \eqref{bnd3}.
The right-hand inequality of \eqref{bnd3} holds due to the Courant-Fischer minimax theorem \cite[Theorem 8.1.2]{gv2} for Hermitian matrices.
\end{proof}

Note that by Algorithm \ref{rQB} and Theorem \ref{thm:2.2},
 $\tilde U_i^TMM^T\tilde U_i=\tilde Z_i^TBB^T\tilde Z_i={\rm diag}(\tilde \sigma_1,\ldots, \tilde \sigma_i)$ with $B=Q^TM$, and $\tilde Z_i$ being the  matrix  containing the first $i$ columns of $\tilde Z$. Here $\tilde \sigma_i=\sigma_i(B)=\|\tilde u_i^TM\|_2$ is close to $\sigma_i$, if $M$ has fast decaying rate singular values. The smaller the index $i$, the more accurate the approximation becomes. The term  \(\|\tilde{u}_i^T M\|_2^2\)  facilitates a simple and direct way in estimating the singular values of $M$. 
Nonetheless, by using $Q$ alone, better prior estimates of the dominant singular values of $M$ can still be obtained, as substantiated by \cite[Theorem 4.3]{gu}:
\begin{equation}
\sigma_i \geq \tilde\sigma_i=\sigma_i(Q^T M) \geq \frac{\sigma_i}{\sqrt{1 + \gamma_i^{4q+2} \|G_{n-k}G_k^\dag\|_2^2}},\label{gu}
\end{equation}
where the gap between the lower and upper bounds is narrower.

\begin{corollary} \label{cor:2.2} With the notation of Theorem \ref{thm:2.2} and for $0<\delta<1$,
define 
\begin{equation}
\mathcal{C}_{n,k,\ell}^\delta=\frac{e \sqrt{\ell}}{\ell-k+1}\left(\frac{2}{{\rm \delta}}\right)^{\frac{1}{\ell-k+1}}\left(\sqrt{n-k}+\sqrt{\ell}+\sqrt{2 \log \frac{2}{{\rm \delta}}}\right).\label{prob}
\end{equation}
Then, with probability at least $1-\delta$,
\begin{eqnarray*}
&{}& \|M-\tilde U_i\tilde U_i^TM\|_F^2\le  \big(\sum_{j>i}\sigma_j^2\big)+\gamma_i^{4q}({\cal C}_{n,k,\ell}^\delta)^2\big(\sum_{j>k}\sigma_j^2\big),\label{bnd10}\\
&{}& \|M-\tilde U_i\tilde U_i^TM\|_2^2\le \sigma_{i+1}^2+\gamma_i^{4q}({\cal C}_{n,k,\ell}^\delta)^2\big(\sum_{j>k}\sigma_j^2\big),\label{bnd20}\\
&{}& |\tilde\sigma_i-\sigma_i|\le  {1\over 2}\gamma_i^{4q+2} ({\cal C}_{n,k,\ell}^\delta)^2{\sigma_i}.\label{bnd30}
\end{eqnarray*}
where $\tilde\sigma_i$ denotes the estimated singular value of $M$ through the RSVD of $M$.

\end{corollary}
\begin{proof} By \eqref{Delta},  we know that 
\[ 
\Delta^2\le \gamma_i^{4q}\|\Sigma_{k,\bot}\|_F^2\|G_{n-k}G_k^\dag\|_2^2\le \gamma_i^{4q}({\cal C}_{n,k,\ell}^\delta)^2\|\Sigma_{k,\bot}\|_F^2
\]
holds with a probability at least $1-\delta$. Here the upper bound $\|G_{n-k}G_k^\dag\|_2\le {\cal C}_{n,k,\ell}^\delta$ is derived from \cite[Theorem 5.8]{gu}. The estimate for $|\tilde\sigma_i-\sigma_i|$ follows from \eqref{gu}.
\end{proof}

\subsection{Randomized t-SVD with fixed threshold and truncated term}

With $K$-truncated term and Fourier-domain singular value threshold, in Algorithm \ref{alg:2.2}, we generalize the idea of the  RSVD algorithm to the tensors. The approximate tensor and singular value tube fibers satisfy the following error bound.

\begin{algorithm}
\caption{Fixed-threshold truncated randomized t-SVD}\label{alg:2.2}

\begin{algorithmic}[1]
\REQUIRE A tensor ${\cal A}\in {\mathbb R}^{n_1\times n_2\times n_3}$, threshold $\tau$ and a target truncated term $K< \min\{n_1, n_2\}$, two integers $p$ and $q>0$ .

\ENSURE Numerical tubal rank $\nu$ and multi-rank vector ${\bf k}$, and the    multi-rank $\bf k$   approximation $\widetilde {\cal A}_K$.

\STATE Compute $\widehat {\cal A}={\sf fft}_3({\cal A})$ and generate a random Gaussian matrix  ${G}\in {\mathbb R}^{n_2\times (K+p)}$.

\FOR {$i=1,2,\ldots, \lceil {n_3+1\over 2}\rceil$}

\STATE  Compute $Y=\Big(\widehat A(:,:,i)\widehat A(:,:,i)\Big)^q\widehat A(:,:,i){G}$.

\STATE  Compute $Q_i={\sf orth}(Y)$

\STATE  Compute $B=Q_i^H\widehat {\cal A}(:,:,i)$.

\STATE  Compute the SVD of $B$ as $B=ZSV^H$.

\STATE  Form $U_K, V_K$ and $S_K$ by truncating $Q_iZ$, $V$ and $S$ with the given $K$, and obtain the numerical rank $k_i$ within the threshold $\tau$.

\STATE  Set $\widehat {\cal U}_K(:,:,i)=U_K, \widehat {\cal V}_K(:,:,i)=V_K$ and $\widehat {\cal S}_K(:,:,i)={\sf bdiag}(S_K(1:k_i,1:k_i), 0_{K-k_i})$.\\

\ENDFOR

\FOR {$i=\lceil {n_3+1\over 2}\rceil+1,\ldots, n_3$}

\STATE  Set $\widehat {\cal U}_K(:,:,i)={\rm conj}(\widehat {\cal U}_K(:,:,n_3-i+1))$, $\widehat {\cal V}_K(:,:,i)={\rm conj}(\widehat {\cal V}_K(:,:,n_3-i+1))$,

\qquad\quad and $\widehat S_K(:,:,i)=\widehat S_K(:,:,n_3-i+1)$.

\ENDFOR

\STATE  Compute $\widetilde {\cal U}_K={\sf ifft}_3(\widehat {\cal U}_K)$, $\widetilde  {\cal V}_K={\sf ifft}_3(\widehat {\cal V}_K)$, $\widetilde  S_K={\sf ifft}_3(\widehat S_K)$
 and $\widetilde {\cal A}_K=\widetilde {\cal U}_K*\widetilde {\cal S}_K*\widetilde {\cal V}_K^T$.

\STATE  Return $\widetilde {\cal A}_K$, the multi-rank vector ${\bf k}=[k_1, k_2,\ldots, k_{n_3}]$ and tubal rank $\nu=\max_i k_i$.
\end{algorithmic}
\end{algorithm}

\begin{theorem}\label{thm:2.3} Suppose that ${\cal A}\in {\mathbb R}^{n_1\times n_2\times n_3}$  and
 $\widehat \sigma_j^{(i)}$ is the $j$-th largest singular value of $\widehat {\cal A}^{(i)}$.
Let
$\widetilde {\cal A}_K$ and the computed multi-rank vector ${\bf k}=[ k_1, k_2,\ldots, k_{n_3}]$ be obtained from Algorithm \ref{alg:2.2}, then for $0< \delta<1$,
\[
\begin{array}l
\|{\cal A}-\widetilde {\cal A}_K\|_2^2\le
\max\limits_{1\le i\le n_3}\Big[(\widehat \sigma^{(i)}_{  k_i+1})^2+\big(\widehat \gamma_{  k_i}^{(i)}\big)^{4q}{  \textsf{C}}_\delta^2
\Big(\sum\limits_{j>K}(\widehat \sigma^{(i)}_{j})^2\Big)\Big],\\
\|{\cal A}-\widetilde {\cal A}_K\|_F^2\le
{1\over n_3}\sum\limits_{i=1}^{n_3}\Big(\sum\limits_{j> k_i}(\widehat \sigma^{(i)}_{j})^2+{\textsf C}_\delta^2\big(\widehat\gamma_{  k_i}^{(i)}\big)^{4q}\sum\limits_{j>K}(\widehat \sigma^{(i)}_{j})^2\Big),\\
\|{\bf s}_i^{\rm tub}-\tilde {\bf s}_i^{\rm tub}\|_2\le{1\over 2 }\widehat\gamma_{i,{\rm max}}^{4q+2} {\cal C}_{\delta}^2\|{\bf s}_i^{\rm tub}\|_2,\quad 1\le i\le \max\limits_{1\le l\le n_3} k_{l}
\end{array}
\]
hold with a probability at least $1-\delta$,
where  ${\textsf C}_\delta:={\cal C}_{n_2,K,K+p}^\delta$ is defined in \eqref{prob}, ${\bf s}_i^{\rm tub}$, $\tilde {\bf s}_i^{\rm tub}$ are the $i$-th singular value tube fiber of  ${\cal A}$ and $\widetilde {\cal A}_K$, respectively.
 The quantity $\widehat\gamma_{ j}^{(i)}=\widehat\sigma_{K+1}^{(i)}/\widehat\sigma_{ j}^{(i)}$, $\widehat\gamma_{j,{\rm max}}=\max\limits_{1\le i\le n_3}\widehat\gamma_{j}^{(i)}$.
\end{theorem}

\begin{proof}
Let ${\cal H}={\sf fft}_3(\widetilde {\cal A}_K)$. Then for $\xi=2$ or $F$, 
\begin{equation}
\|{\cal A}-\widetilde {\cal A}_K\|_\xi^2=\rho_\xi^2\|\widehat  {\cal A}-\widehat {\cal H}\|_\xi^2=\rho_\xi^2 \cdot {\mathbb D}_\xi\left(\|\widehat {\cal A}^{(i)}-\widehat {\cal H}^{(i)}\|_\xi^2\right),\label{er}
\end{equation}
where    $\rho_\xi=1$ for $\xi=2$, and $\rho_\xi={1\over {n_3}}$ otherwise,
${\mathbb D}_2(\cdot), {\mathbb D}_F(\cdot)$ are the maximal and summation operators for $i$ ranging from 1 to $n_3$, respectively. The matrix
$\widehat {\cal H}^{(i)}= U_{k_i} U_{k_i}^T\widehat {\cal A}^{(i)}=\widehat{\cal U}_K^{(i)}\widehat{\cal S}_K^{(i)}\widehat{\cal V}_K^{(i)^H}$, where    $U_{k_i}$ contains the first $k_i$ columns of $U_{K}$, with $U_K$ defined in  Line 7 of Algorithm \ref{alg:2.2}.

By Corollary \ref{cor:2.2}   we know that
\[
\begin{array}l
\|\widehat {\cal A}^{(i)}-\widehat {\cal H}^{(i)}\|_2^2\le
(\widehat \sigma_{k_i+1}^{(i)})^2+(\widehat\gamma_{k_i}^{(i)})^{4q}{\sf C}_\delta^2\big(\sum_{j>K}(\widehat\sigma_j^{(i)})^2\big),\\
\|\widehat {\cal A}^{(i)}-\widehat {\cal H}^{(i)}\|_F^2\le
 \sum\limits_{j>k_i}(\widehat\sigma_j^{(i)})^2+(\widehat\gamma_{k_i}^{(i)})^{4q}{\sf C}_\delta^2\big(\sum_{j>K}(\widehat\sigma_j^{(i)})^2\big),
\end{array}
\]
and for  $1\le i\le \max\limits_{1\le l\le n_3} k_{l}$, 
\begin{align*}
\|{\bf s}_i^{\rm tub}-\tilde {\bf s}_i^{\rm tub}\|_2&={1\over \sqrt{n_3}}\|\widehat{\bf s}_i^{\rm tub}-\widehat{\tilde {\bf s}_i^{\rm tub}}\|_2={1\over \sqrt{n_3}}\big[\sum\limits_{l=1}^{n_3}
(\widehat\sigma_i^{(l)}-\widehat{\tilde\sigma}_i^{(l)})^2\big]^{1/2}\\
&\le {1\over 2 \sqrt{n_3}}\gamma_{i,{\rm max}}^{4q+2} {\sf C}_{\delta}^2[\sum_{l=1}^{n_3}(\widehat{\sigma}_i^{(l)})^2]^{1/2}=
{1\over 2 \sqrt{n_3}}\gamma_{i,{\rm max}}^{4q+2} {\sf C}_{\delta}^2\|\widehat{\bf s}_i^{\rm tub}\|_2,
\end{align*}
where $ \widehat{\tilde\sigma}_i^{(l)}$ is the $i$-th diagonal entry of $ \widehat{\cal S}_K^{(l)}$, 
and  $|\widehat\sigma_i^{(l)}-\widehat{\tilde\sigma}_i^{(l)}|=|\widehat\sigma_i^{(l)}|$ for $i>k_l$.
 Thus we get the desired   result from Definition \ref{def:2.2} and the relation  $\|\widehat{\bf s}_i^{\rm tub}\|_2=\sqrt{n_3}\|{\bf s}_i^{\rm tub}\|_2$.
\end{proof}

\begin{remark} To ensure a good t-SVD approximation within the threshold \(\tau\), it is essential to appropriately predict $K$ in Algorithm \ref{alg:2.2}. Typically, \({k}_i \leq K\). If $K$ is excessively small, then \({k}_i = K\), and some singular values \(\widehat{\sigma}_j^{(i)}\) of \(\widehat{\cal A}^{(i)}\) that are bigger than \(\tau\) but have an index \(j > K\) may be neglected. This could lead to a significant approximation error for \(\widehat{\cal A}^{(i)}\).Thus, an algorithm that determines the numerical multi-rank using a threshold \(\tau\), while doing away with the need to specify the parameter $K$, is highly necessary.

\end{remark}

\section{Randomized tubal rank-revealing  algorithm and tensor  approximation}

Building upon the framework in \cite{liu}, the adaptive tubal rank-revealing algorithm constructs an orthonormal basis matrix \( Q := [Q~~Q_j] \) to approximate the range of \( M \in \mathbb{R}^{m \times n} \) with a fixed threshold \(\tau\). Starting from an empty \( Q \), each iteration generates an \( n \times b \) Gaussian matrix \(G_j\) and computes \(\bar{Q}_j = \text{orth}((I - QQ^T)MG_j)\) to align \( \text{Ran}(Q_j) \) with \( \text{Ran}(U_j) \), where $U_j\in {\mathbb R}^{m\times b}$ is the \( j \)-th block of \( M \)'s left singular matrix \(U = [U_1~U_2~\ldots~U_h]\). The iteration stops when a singular value of \( \bar{Q}_j^T M \) drops below \(\tau\). To enhance accuracy, a power scheme \(\bar{Q}_j = \text{orth}((\bar{M}\bar{M}^T)^q \bar{M}G_j)\) (where \(\bar{M} = (I - QQ^T)M=:{\cal P}_{Q}^\bot M\)) can be used; alternatively, \(\bar{Q}_j = \text{orth}({\cal P}_{Q}^\bot(MM^T{\cal P}_{Q}^\bot)^q MG_j)\) avoids explicit formulation of \(\bar{M}\) and thus reduces computational cost. The algorithm refines \( Q_j \) using Rayleigh-Ritz acceleration: compute \( Q_j := \bar{Q}_j \tilde{U}_j \) (with \(\tilde{U}_j\) from the left singular matrix of \( \bar{Q}_j^T M \)) and orthogonalize via \( Q_j := {\sf orth}((I - QQ^T)Q_j)\).

Another strategy to enhance the quality of \(Q_j\) involves introducing oversampled columns in \(G_j\). For the restricted range \({\rm Ran}_\tau(M)\)—spanned by the top $k$ left singular vectors of $M$ with corresponding singular values exceeding \(\tau\)—the authors of \cite{liu} show that if the subspace deviation between \({\rm Ran}(Q)\) and  \({\rm Ran}_\tau(M)\) is less than \(\epsilon\), the deflated matrix \((I - QQ^T)M\) will preserve the singular values of $M$ with an absolute error bounded by \(\epsilon \|M\|_2\). This validates using an initially fixed threshold for numerical rank detection via deflation. Moreover, the oversampling strategy is often unnecessary in practice, as the range of the previously generated basis matrix already lies within \({\rm Ran}_\tau(M)\) with small error. In \cite{liu}, theoretical results are provided for scenarios where the numerical rank is a multiple of $b$.

We can apply the above adaptive scheme to each frontal slice of the tensor \(\mathcal{A}\), say \(\mathcal{A}^{(i)}\), to obtain the approximate basis matrix \(Q^{(i)} \in \mathbb{R}^{m \times k_i}\) for the range \(\text{Ran}(\widehat{\mathcal{A}}^{(i)})\) within the specified threshold. This allows us to determine the multi-rank vector of \(\mathcal{A}\). The detailed algorithm is presented in Algorithm \ref{t-ARTRR-I}. However, for theoretical guarantees, we must address the case where \(k_i\) may not be a multiple of $b$.

\begin{algorithm}\caption{({\sf r-TuRank}) Adaptive randomized tubal rank-revealing  algorithm for tensor decomposition within given threshold}  \label{t-ARTRR-I}
\begin{algorithmic}[1]
\REQUIRE A tensor ${\cal A}\in {\mathbb R}^{n_1\times n_2\times n_3}$, block size $b$, threshold  $\tau$ and power scheme parameter $q>0$ .

\ENSURE Computed multi-rank $\{k_i\}_{i=1}^{n_3}$,  tubal rank $\nu$ and  tubal rank-$\nu$ approximation $\widetilde {\cal A}_\nu$.

\STATE  Compute $\widehat {\cal A}={\sf fft}_3({\cal A})$.

 \FOR { $i=1,2,3,\ldots, \lceil {n_3+1\over 2}\rceil$}

\FOR{ $j=1,2,\ldots$}

\STATE Initialize $Q^{(i)}=[~]$.

\STATE Set $G_j^{(i)}={\sf randn}(n_2,b)$, $Y_0={\sf orth}(\widehat {\cal A}(:,:,i)G_j^{(i)})$.

\FOR {$h=1,2,\ldots,q$}

\STATE Compute $Y_{h-1}:=Y_{h-1}-Q^{(i)}(Q^{(i)^H}Y_{h-1})$.

\STATE  Compute $Z={\sf orth}(\widehat{\cal A}(:,:,i)^HY_{h-1})$.

\STATE Compute $Y_h={\sf orth}(\widehat{\cal A}(:,:,i)Z)$.

\ENDFOR

\STATE  Compute $\tilde Q_j^{(i)}={\sf orth2}(Y_q-Q^{(i)}(Q^{(i)^H}Y_q))$.

\STATE Compute $T=\tilde Q_j^{(i)^H}\widehat{\cal A}(:,:,i)$ and its SVD: $T=\tilde Z_j^{(i)}\tilde S_j^{(i)}\tilde V_j^{(i)^H}$. Set $Q_j^{(i)}=\tilde Q_j^{(i)}\tilde Z_j^{(i)}$.

\IF  {there exists  a smallest index $t\le b$ with {$\tilde S_j(t,t)<\tau$}}

\STATE Set $Q_j^{(i)}(:,t:b)=[~]$ and $k_i=(j-1)b+t-1$.

\STATE  Goto step 17, and then break the $j$-loop.

\ENDIF

\STATE Compute $Q_j^{(i)}={\sf orth}\big((I-Q^{(i)}Q^{(i)^H})Q_j^{(i)}\big)$ and set $Q^{(i)}:=[Q^{(i)} \quad Q_j^{(i)}]$.

\STATE Compute ${\widetilde {\cal A}}(:,:,i)=Q^{(i)}Q^{(i)^H}\widehat {\cal A}(:,:,i)$.

\ENDFOR

\FOR {$i=\lceil {n_3+1\over 2}\rceil+1,\ldots, n_3$}

\STATE Set $\widetilde {\cal A}(:,:,i)={\rm conj}(\widetilde {\cal A}(:,:,n_3-i+2))$ and $k_i= k_{n_3-i+2}$.

\ENDFOR

\ENDFOR

\RETURN   $\{k_i\}_{i=1}^{n_3}$,  $\nu=\max_i k_i$ and   $\widetilde{\cal A}_\nu={\sf ifft}_3(\widetilde{\cal A})$.
\end{algorithmic}

\end{algorithm}

\subsection{Theoretical results}
To analyze the error bound of the approximate tensor $\widetilde{\cal A}_\nu$, we first introduce a lemma and some notations.

In Algorithm \ref{t-ARTRR-I}, suppose that  $k_i=b(s_i-1)+\rho_{s_i}$($1\le \rho_{s_i}\le b$) is the computed  rank  of
the frontal slice $\widehat {\cal A}^{(i)}$, and
the associated SVD of $\widehat {\cal A}^{(i)}$ is  $\widehat {\cal A}^{(i)}=\widehat {U}^{(i)}\widehat {\Sigma}^{(i)}\widehat {V}^{(i)^H}$ with $\widehat { \Sigma}^{(i)}={\rm diag}(\widehat \sigma_1^{(i)}, \ldots, \widehat \sigma_{\nu_i}^{(i)},0,\ldots,0)$ with $\nu_i={\rm rank}(\widehat A^{(i)})$. Denote
$\widehat { U}^{(i)}$ and $\widehat {U}^{(i)}_{[\ell]}$  as
\begin{equation}
\begin{array}l
\widehat {U}^{(i)}=[\widehat {U}_1^{(i)}~~\widehat {  U}_2^{(i)}~ \ldots~\widehat {  U}_{s_i-1}^{(i)}~~ \widehat {  U}_{s_i}^{(i)}], \qquad
\widehat {U}^{(i)}_{[\ell]}=[\widehat {U}_1^{(i)}~~\widehat {  U}_2^{(i)}~ \ldots~ \widehat {  U}_{\ell}^{(i)}].\\
\qquad\quad ~~ b~~~~~b~~~~~\ldots~~~b~~~~~~~b~~~~
\end{array}\label{eq4.3}
\end{equation}
 Let   $\tilde Q^{(i)}\in {\mathbb C}^{n_1\times bs_i}$, $\tilde Z_{\ell}^{(i)}$ and $Q^{(i)}\in {\mathbb C}^{n_1\times k_i}$ be computed from Algorithm \ref{t-ARTRR-I}.
  Partition 
\begin{eqnarray}
 &{}&\tilde {  Q}^{(i)}=[ {\tilde  Q}_1^{(i)}~~  {\tilde Q}_2^{(i)}~ \ldots~  {\tilde  Q}_{s_i-1}^{(i)}~~   {\tilde  Q}_{s_i}^{(i)} ],\quad {\tilde  Q}_{[\ell]}^{(i)}=[ {\tilde  Q}_1^{(i)}~~ {\tilde  Q}_2^{(i)}~ \ldots~   { \tilde Q}_{\ell}^{(i)}],\label{eq4.1}\\
 &{}&Q^{(i)}=[Q_1^{(i)}~\ldots~ Q_{s_i-1}^{(i)} ~ ~ Q_{s_i}^{(i)}]\in {\mathbb C}^{n_1\times k_i}\quad \mbox{with}\quad Q_{\ell}^{(i)}=\tilde Q_{\ell}^{(i)}\tilde Z_{\ell}^{(i)}(:,1:\rho_{s_i}),\label{Qi}
\end{eqnarray}
to make the subblock size be conformal with those of $\widehat {U}^{(i)}$ and $\widehat {U}_{[\ell]}^{(i)}$, respectively, in which 
$\tilde Q_\ell^{(i)}\in {\mathbb C}^{n_1\times b}$ is   computed from $\tilde Q_\ell^{(i)}={\sf orth}\big((\tilde A_{\ell-1}^{(i)}\tilde A_{\ell-1}^{(i)^H})^q\tilde A_{\ell-1}^{(i)}G_\ell^{(i)}\big)$ with
${  Q}^{(i)}_{[0]}:=0$ and 
\begin{equation}
\tilde A^{(i)}_{\ell}=(I- {  Q}^{(i)}_{[\ell-1]}  {  Q}^{(i)^H}_{[\ell-1]}) \widehat{\cal A}^{(i)},\quad 1\le \ell\le s_i.\label{eq4.2}\\
\end{equation}
The subblocks $Q_l^{(i)}(1\le l\le s_i-1)$ has $b$ columns, while   the last subblock
$Q_{s_i}^{(i)}$  has  $\rho_{s_i}$ (possibly fewer than $b$)  columns.
Obviously, $\tilde Q_\ell^{(i)^H}{  Q}^{(i)}_{[\ell-1]}=0$ and $\tilde Q_\ell^{(i)^H}\tilde A^{(i)}_{\ell}=\tilde Q_\ell^{(i)^H}\widehat{\cal A}^{(i)}.$

Let $\tilde \sigma_{\ell,j}^{(i)}$ be the $j$-th singular value of $\tilde A^{(i)}_{\ell}$. Then by  \cite[Theorems 3.6]{liu}, 
 \begin{equation}
\begin{array}{ll}
 | \tilde \sigma_{\ell,j}^{(i)}-\widehat\sigma_{b(\ell-1)+j}^{(i)}|\le\varepsilon_{\ell-1}^{(i)} \|\widehat {\cal A}^{(i)}\|_2, &j=1,\ldots,n_2-b(\ell-1),\\
\tilde \sigma_{\ell,j}^{(i)}\le \varepsilon_{\ell-1}^{(i)} \|\widehat {\cal A}^{(i)}\|_2,&j=n_2-b(\ell-1)+1,\ldots, n_2,
\end{array}\label{sv}
\end{equation} 
where  
\[
\varepsilon_{j}^{(i)}=\| { Q}^{(i)}_{[j]}{ Q}^{(i)^H}_{[j]}-\widehat {  U}^{(i)}_{[j]}\widehat { U}^{(i)^H}_{[j]}\|_2,\quad 1\le j\le \ell-1
\]
measures the distance between ${\rm Ran}( {Q}^{(i)}_{[j]})$ and ${\rm Ran}(\widehat {U}^{(i)}_{[j]})$.

The theory below presents the approximation error bounds for the case that $k_i$ might not be a multiple of $b$.

\begin{theorem} \label{thm:4.2} With the notations and the results  in \eqref{eq4.3}-\eqref{sv}, if $\varepsilon_{s_i-1}^{(i)}<1$ and $\widehat\sigma_{k_i}^{(i)}-\widehat\sigma_{bs_i+1}^{(i)}>2\varepsilon_{s_i-1}^{(i)}\|\widehat{\cal A}^{(i)}\|_2$ for $i=1,2,\ldots, n_3$,
then for $0<\delta<1$, the computed tubal rank-$\nu$ approximation $\widetilde {\cal A}_\nu$   from Algorithm \ref{t-ARTRR-I} satisfies the error  bound
\[
\begin{array}l
\|{\cal A}-\widetilde {\cal A}_\nu\|_2^2\le \max\limits_{1\le i\le n_3}\Big[ \big(\tilde \sigma^{(i)}_{s_i,\rho_{s_i}+1}\big)^2+   (\tilde\gamma_{s_i,\rho_{s_i}}^{(i)})^{4q}
{\sf C}_\delta^2\sum\limits_{j>b}(\tilde \sigma^{(i)}_{j})^2\Big],\\[8pt]
\|{\cal A}-\widetilde {\cal A}_\nu\|_F^2\le {1\over n_3}\sum\limits_{i=1}^{n_3}\Big[\sum\limits_{j>\rho_{s_i}}(\tilde \sigma^{(i)}_{s_i,j})^2+(\tilde\gamma_{\rho_{s_i}}^{(i)})^{4q}{\sf C}_\delta^2\sum\limits_{j>b}(\tilde \sigma^{(i)}_{s_i,j})^2\Big],
\end{array}
\]
with   probability at least $1-\delta$, where  ${\sf C}_\delta:={\cal C}_{n_2,b, b}^\delta$ is
 defined in \eqref{prob}, and
 \[
\tilde\gamma_{s_i,\rho_{s_i}}^{(i)}={\tilde\sigma_{s_i,b+1}^{(i)}\over \tilde\sigma_{s_i,\rho_{s_i}}^{(i)} }\le {\widehat\sigma_{bs_i+1}^{(i)}+\varepsilon_{s_i-1}^{(i)} \|\widehat {\cal A}^{(i)}\|_2\over \widehat \sigma_{k_i}^{(i)}-\varepsilon_{s_i-1}^{(i)} \|\widehat {\cal  A}^{(i)}\|_2}<1.
 \]
 \end{theorem}

\begin{proof}
Let ${\cal H}={\sf fft}_3(\widetilde {\cal A}_\nu)$. It should be noted that the relation in \eqref{er} still holds,
in which   $\widehat {\cal H}^{(i)}=Q^{(i)}Q^{(i)^H}\widehat {\cal A}^{(i)}$ is the rank-$k_i$ approximation to $\widehat {\cal A}^{(i)}$, within the range of $Q^{(i)}$. Notably,  the last subblock $Q_{s_i}^{(i)}$ of $Q^{(i)}$  is obtained by applying RSVD (Algorithm \ref{rQB}) on the matrix $\tilde A_{s_i}^{(i)}$ without oversampling, and thus we get from Corollary \ref{cor:2.2} that
\[
\begin{array}l
\|(I-Q_{s_i}Q_{s_i}^H)\tilde A_{s_i}^{(i)}\|_2^2\le \big(\tilde \sigma^{(i)}_{s_i,\rho_{s_i}+1}\big)^2+   (\tilde\gamma_{s_i,\rho_{s_i}}^{(i)})^{4q}
{\sf C}_\delta^2\sum\limits_{j>b}(\tilde \sigma^{(i)}_{j})^2,\\
\|(I-Q_{s_i}Q_{s_i}^H)\tilde A_{s_i}^{(i)}\|_F^2\le \sum\limits_{j>\rho_{s_i}}(\tilde \sigma^{(i)}_{s_i,j})^2+(\tilde\gamma_{\rho_{s_i}}^{(i)})^{4q}{\sf C}_\delta^2\sum\limits_{j>b}(\tilde \sigma^{(i)}_{s_i,j})^2,
\end{array}
\]
where   by the orthogonality between $Q_{s_i}$ and $Q_{[s_i-1]}$, we get
\[
\begin{array}{rl}
(I-Q_{s_i}Q_{s_i}^H)\tilde A_{s_i}^{(i)}&=(I-Q_{s_i}Q_{s_i}^H)(I-Q_{[s_i-1]}Q_{[s_i-1]}^H)\widehat  {\cal A}^{(i)}\\
&=(I-Q^{(i)}Q^{(i)^H})\widehat {\cal A}^{(i)}.
\end{array}
\]
 Notably, the random Gaussian matrix $G_{s_i}^{(i)}=\Big[{G_{s_i,b}^{(i)}\atop G_{s_i,-b}^{(i)}}\Big]{b\atop n_2-b}$ used for computing $\tilde Q_{s_i}^{(i)}$ satisfies  $\Gamma_i:=\|G_{s_i,-b}^{(i)}(G_{s_i,b}^{(i)})^\dag\|_2\le {\sf C}_\delta$ with probability at least $1-\delta$. When we vary $i$ from 1 to $n_3$, the probability of all $n_3$ random Gaussian matrices satisfies  the same upper bound ($\Gamma_i\le {\sf C}_\delta$ for all $i$) remains $1-\delta$. This is because the symmetry enforces a uniform probabilistic guarantee across the entire collection, overriding the product rule for independent events in this specific setting.
We therefore get the estimates for $\|{\cal A}-\widetilde {\cal A}_\nu\|_2$ and $\|{\cal A}-\widetilde {\cal A}_\nu\|_F$,  and
the upper bound for
$\tilde\gamma_{s_i,\rho_{s_i}}^{(i)}$ follows from \eqref{sv}.
\end{proof}

\begin{theorem} \label{thm:3.2} Adopting the notation and the results  from \eqref{eq4.3}-\eqref{sv},  let $\sigma_j(\tilde Q_l^{(i)^H}\widehat {\cal A}^{(i)})$ (where $1\le i\le n_3, 1\le l\le s_i$, $1\le j\le b$) be the  $(b(l-1)+j)$-th diagonal entry of the $i$-th frontal slice $(\widehat {\cal S}^{est})^{(i)}$ of an  $f$-diagonal tensor $\widehat {\cal S}^{est}$. Define ${\cal S}^{est}={\sf ifft}_3(\widehat{\cal S}^{est})$. Then $ {\cal S}^{est}$ 
can be interpreted as the estimated singular-value tensor
 ${\cal S}$ corresponding to   ${\cal A}$, and their singular value tube fibers satisfy  the error bound
\begin{equation}
\|{\bf s}_{b(l-1)+j}^{\rm {\it est},tub}-{\bf s}_{b(l-1)+j}^{\rm tub}\|_2^2\le 2\tau_{j,{\rm max}}^2\|{\bf s}_{b(l-1)+j}^{\rm tub}\|_2^2+2\tau_\varepsilon^2\|{\cal A}\|_2^2,\label{tub1}
\end{equation}
where   $\tau_{j,{\rm max}}={1\over 2}\tilde\gamma_{j,\max}^{4q+2}{\sf C}_\delta^2$ with ${\sf C}_\delta={\cal C}_{n_2,b,b}^\delta$ and 
$\tilde\gamma_{j,\max}=\max\limits_{1\le i\le n_3}\tilde\sigma_{l,b+1}^{(i)}/\tilde\sigma_{l,j}^{(i)}$, and $$\tau_\varepsilon=
(1+\tau_{j,\max})\max\limits_{1\le i\le n_3}\varepsilon_{l-1}^{(i)}.$$
\end{theorem}

\begin{proof} Recall that  $\tilde\sigma_{l,j}^{(i)}$ denotes the $j$-th singular value of 
$\tilde A_l^{(i)}$. By Corollary \ref{cor:2.2}, for $1\le j\le b$,  we 
get  
\[
|\sigma_j(\tilde Q_{l}^{(i)^H}\tilde {A}_{l}^{(i)})-\tilde\sigma_{l,j}^{(i)}|=|\sigma_j(Q_{l}^{(i)^H}\widehat {\cal A}^{(i)})-\tilde\sigma_{l,j}^{(i)}|\le {1\over 2}\tilde\gamma_j^{4q+2}{\sf C}_\delta^2\tilde\sigma_{l,j}^{(i)}.
\]
Combining it and the triangular inequality with \eqref{sv}, we observe
\[
|\sigma_j(Q_{l}^{(i)^H}\widehat {\cal A}^{(i)})-\widehat\sigma_{b(l-1)+j}^{(i)}|\le {1\over 2}\tilde\gamma_j^{4q+2}{\sf C}_\delta^2\tilde\sigma_{l,j}^{(i)}+\varepsilon_{l-1}^{(i)}\|\widehat{\cal A}^{(i)}\|_2\le 
\tau_{j,\max}\tilde\sigma_{l,j}^{(i)}+\varepsilon_{l-1}^{(i)}\|\widehat{\cal A}^{(i)}\|_2, 
\]
where $ \tilde\sigma_{l,j}^{(i)}\le \widehat\sigma_{b(l-1)+j}^{(i)}+\varepsilon_{l-1}^{(i)}\|\widehat{\cal A}^{(i)}\|_2$. Thus giving
\[
|\sigma_j(\tilde Q_{l}^{(i)^H}\widehat {\cal A}^{(i)})-\widehat\sigma_{b(l-1)+j}^{(i)}|\le 
\tau_{j,\max}\widehat\sigma_{b(l-1)+j}^{(i)}+(1+\tau_{j,\max})\varepsilon_{l-1}^{(i)}\|\widehat{\cal A}^{(i)}\|_2.
\]
With a similar argument to proving Theorem \ref{thm:2.3}, we get
\begin{align*}
&\|{\bf s}_{b(l-1)+j}^{\rm {\it est},tub}-{\bf s}_{b(l-1)+j}^{\rm tub}\|_2^2\le {1\over  {n_3}}\sum\limits_{i=1}^{n_3}
(\sigma_j(\tilde Q_{l}^{(i)^H}\widehat {\cal A}^{(i)})-\widehat\sigma_{b(l-1)+j}^{(i)})^2\\
&\le 2\tau_{j,\max}^2\Big({1\over n_3}\sum\limits_{i=1}^{n_3}(\widehat\sigma_{b(l-1)+j}^{(i)})^2\Big)+2\tau_\varepsilon^2\Big({1\over  {n_3}}\sum\limits_{i=1}^{n_3}\|\widehat{\cal A}^{(i)}\|_2^2\Big)\\
&\le 2\tau_{j,\max}^2\|{\bf s}_{b(l-1)+j}^{\rm tub}\|_2^2+2\tau_{\varepsilon}^2\|{\cal A}\|_2^2,\notag
\end{align*}
where in the second inequality, we used the fact that $(x+y)^2\le 2x^2+2y^2$.
\end{proof}

In \cite{liu}, for $q=0$, \(\varepsilon_{l-1}^{(i)}=O({\widehat \sigma_{b(l-1)+1}^{(i)}{\sf C}_\delta'/\widehat \sigma_{b(l-1)}^{(i)}})\)  was estimated, where ${\sf C}_\delta'={\cal C}^\delta_{n_2,b(l-1),b(l-1)}$.  While for $q>0$ we cannot quantify the magnitude of \(\varepsilon_{l-1}^{(i)}\), the term \(\tau_{j,\max}\) within the estimated bound likely ensures that the singular value tube fiber with a smaller index $j$ exhibits a more accurate approximation error. 
This will be  verified in later experimental results.

\section{Numerical experiments}

In this section, we first evaluate the effectiveness  of the proposed algorithm using synthetic tensors.
Then we use the real-world data to examine the performance of TRPCA  via the
applications in computer vision. All the experiments are conducted on a personal laptop with
Intel(R) Core(TM) i5-7200U CPU(2.50GHz) and 8.00 GB RAM, and all the algorithms have been implemented in the MATLAB
R2017b environment and Tensor-Tensor Product Toolbox \cite{lu}.   

\subsection{Synthetic Data}
In this subsection, we present synthetic tensors to test feasibility of
the proposed algorithm in evaluating the numerical tubal rank and approximate tensor.

Let ${\cal A}\in {\mathbb R}^{400\times 400\times 50}$. For $i=1,2,\ldots, 50$, 
let ${\bf U}^{(i)}, {\bf V}^{(i)}\in {\mathbb R}^{400\times 400}$  be random orthogonal matrices.
We define
two types of tensors:

\begin{itemize}
\item Tensor I:   ${\cal A}={\sf ifft}_3(\widehat {\cal A})$, where the frontal slice   $\widehat {\cal A}^{(i)}$ is defined as \cite{kc}: $\widehat {\cal A}^{(i)}={\bf U}^{(i)}\Sigma_1 {\bf V}^{(i)^T}$  with $e^{-k/6}$ as the $k$-th ($1\le k\le 15)$ diagonal entry of the diagonal matrix ${ \Sigma}_1$, and $e^{-k/2}$ otherwise. The neighbouring  ratio of the first 15 diagonal entries is about 0.85, and it decreases to $4\times 10^{-3}$ for $\widehat \sigma_{16}^{(i)}/\widehat \sigma_{15}^{(i)}$, and then keeps 0.61 in the later neighbouring diagonal entries.

\item Tensor II:  ${\cal A}^{(i)}={\bf U}^{(i)}\Sigma_2 {\bf V}^{(i)^T}$, where the $k$-th diagonal entry of ${  \Sigma}_2$ is $2^{-k}$.  It has fast decaying rate in the singular values  of its frontal slices, while  this  does not necessarily guarantee the singular values of   $\widehat {\cal A}^{(i)}$ to have fast decaying rate. Actually, the initial neighbouring singular value ratio of $\widehat{\cal A}^{(i)}$ is about 0.95, and the sequent neighbouring singular value ratio is up to 0.98.
\end{itemize}

\begin{figure}
\centering
\includegraphics[width=0.9\textwidth,height=5cm]{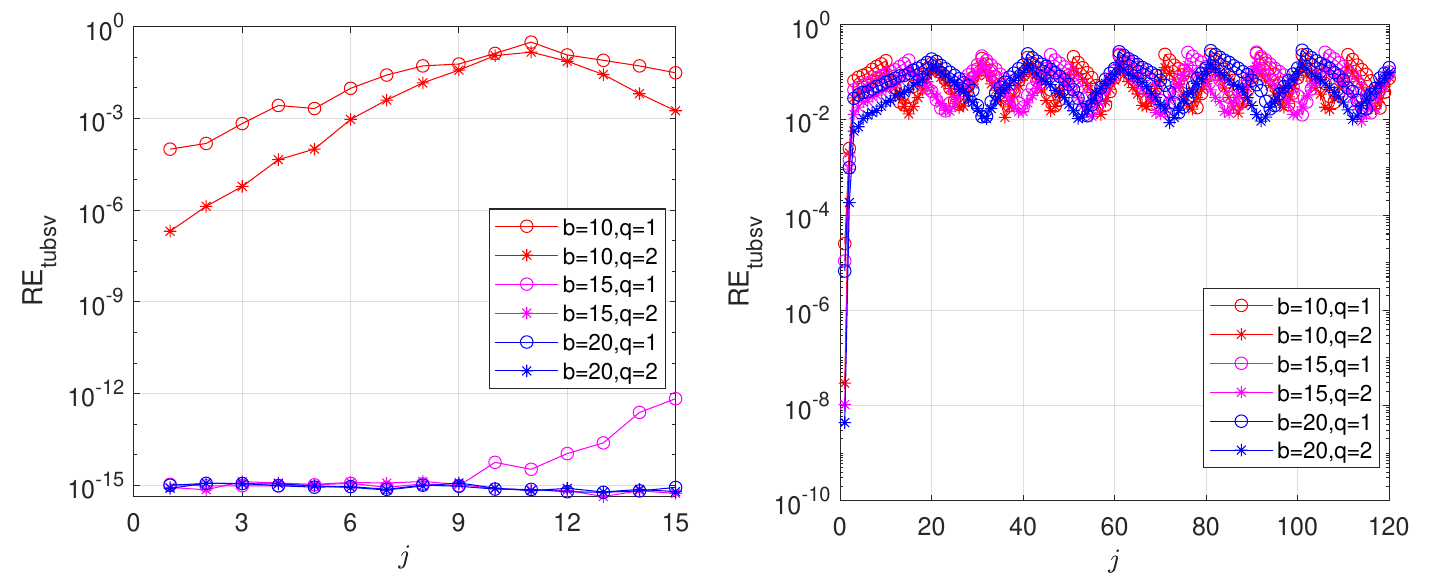}
\caption{\it The relative errors of  computed singular value tube fibers  for different $b$ and $q$. The left figure corresponds to   Tensor I and the right is for   tensor II.} \label{Fig5.1}
 \end{figure}

\begin{figure}
\centering
\includegraphics[width=0.9\textwidth,height=8cm]{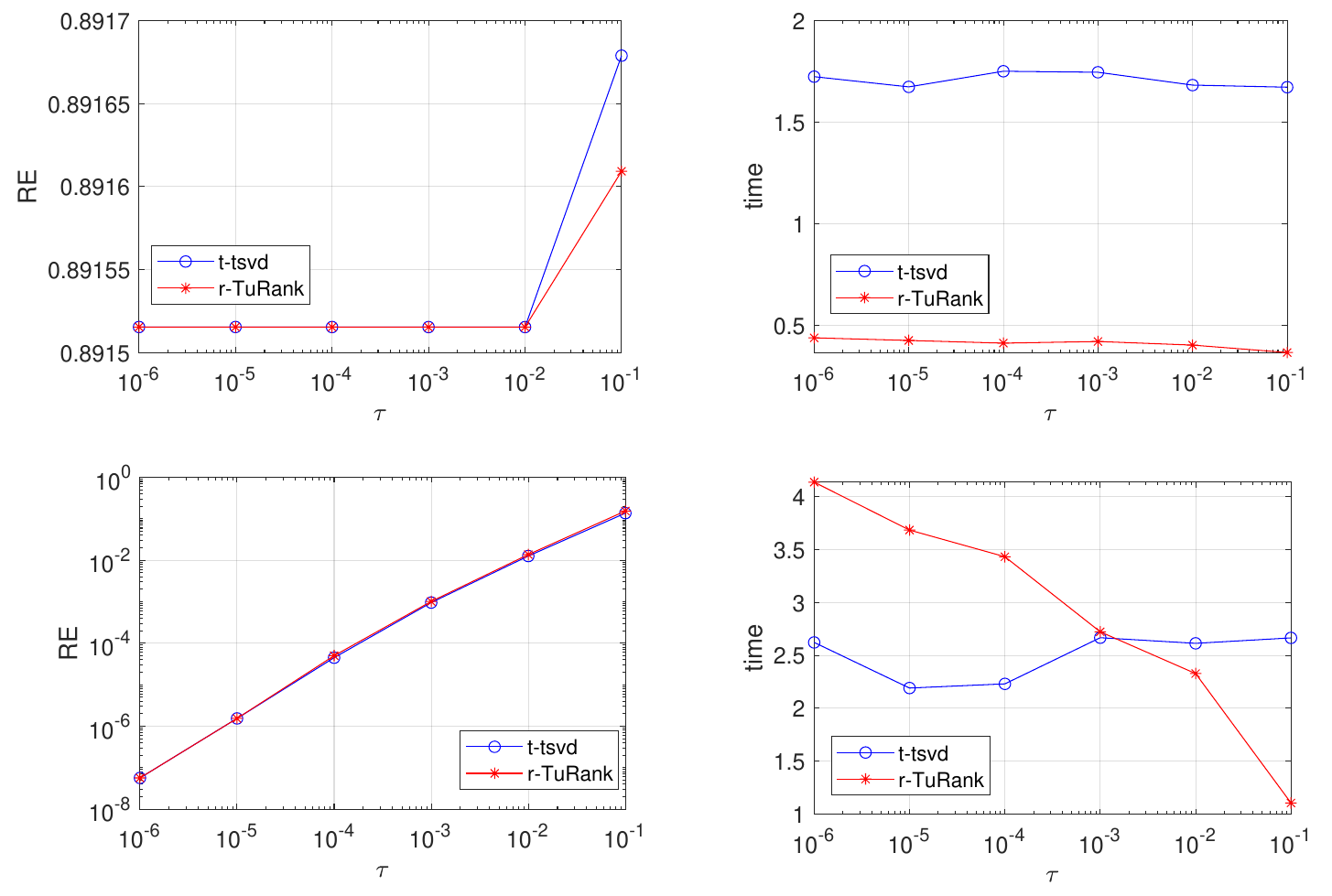}
\caption{\it The Frobenius relative errors of  approximate tensor  for $b=10$ and $q=1$. The top row corresponds to relative error and running  time in seconds for   Tensor I and the bottom is for  tensor II.} \label{Fig5.2}
 \end{figure}

In Fig. \ref{Fig5.1}  we  set the threshold $\tau$ to be $\tau=5\times 10^{-2}$, and  test the relative accuracy 
\[
{\rm RE}_{\rm tubsv}(j)={\displaystyle  \|{\bf s}_j^{est,{\rm tub}}-{\bf s}_j^{\rm tub}\|_2 \over \displaystyle \|{\bf s}_j^{\rm tub}\|_2},\quad  1\le j\le\nu,
\]
of the estimated singular value tube fibers computed from the proposed algorithm.   It is observed that
for fixed $b$,
 the estimated singular value tube fibers ${\bf s}^{est,{\rm tub}}_j$  of Tensor I have better accuracy with a big $q$.
 For fixed $q$,   high accuracy is achieved  when $b$ is taken as 15 or  20, while for $b=10$, the accuracy decreases.
 Moreover, the smaller the index ($j$, where $1\le j\le b$) of the tuber fiber, the more accurate   the relative error.  This coincides with the first term of the estimate \eqref{tub1}. For $b=15$ or $20$, there is a significant big gap between $\widehat\sigma_{15}^{(i)}$
  and $\widehat\sigma_{16}^{(i)}$, while for $b=10$, the first 15 adjacent singular values of $\widehat {\cal S}^{(i)}$ increase up to 0.85,
  this affects  a high-accuracy of the estimation of the singular values. 
  The accuracy improves  when $b\le j\le 2b$, partly due to the singular value ratio $\widehat \sigma_{j+1}^{(i)}/\widehat \sigma_{j}^{(i)}$ decreases to 0.61.

 For Tensor II, it is also observed that the first few singular value tube fibers are of the precision $10^{-8}$, and the accuracy decreases gradually and oscillates between $10^{-2}$ and $10^{-1}$ in the sequent singular value tube fibers. This oscillation is also justified by the estimates in \eqref{tub1}, where  it is seen that the estimated singular value tube fibers have better accuracy for small $j$, and when the index $j$ increases to a multiple of $b$, the accuracy will decrease.

In Fig. \ref{Fig5.2}, we take $b=10, q=1$ and vary the threshold $\tau$  from $10^{-6}$ to $10^{-1}$. we compare the Frobenius relative error and efficiency of the computed approximate tensor $\widetilde{\cal A}$ via the truncated t-SVD ({\sf t-tsvd}) and {\sf r-TuRank} algorithm, respectively, in which
\begin{equation}
{\rm RE}={\|\widetilde{\cal A}_\nu-{\cal A}\|_F/ \|{\cal A}\|_F}.\notag
\end{equation}
The results presented show that the proposed {\sf r-TuRank} algorithm is comparable in accuracy to the truncated t-SVD across different values of \(\tau\). Moreover, for the approximation problems involving fast-decaying Fourier-domain singular value tube fibers, the proposed algorithm exhibits higher efficiency.

\begin{figure}
\centering
\includegraphics[width=0.6\textwidth,height=6cm]{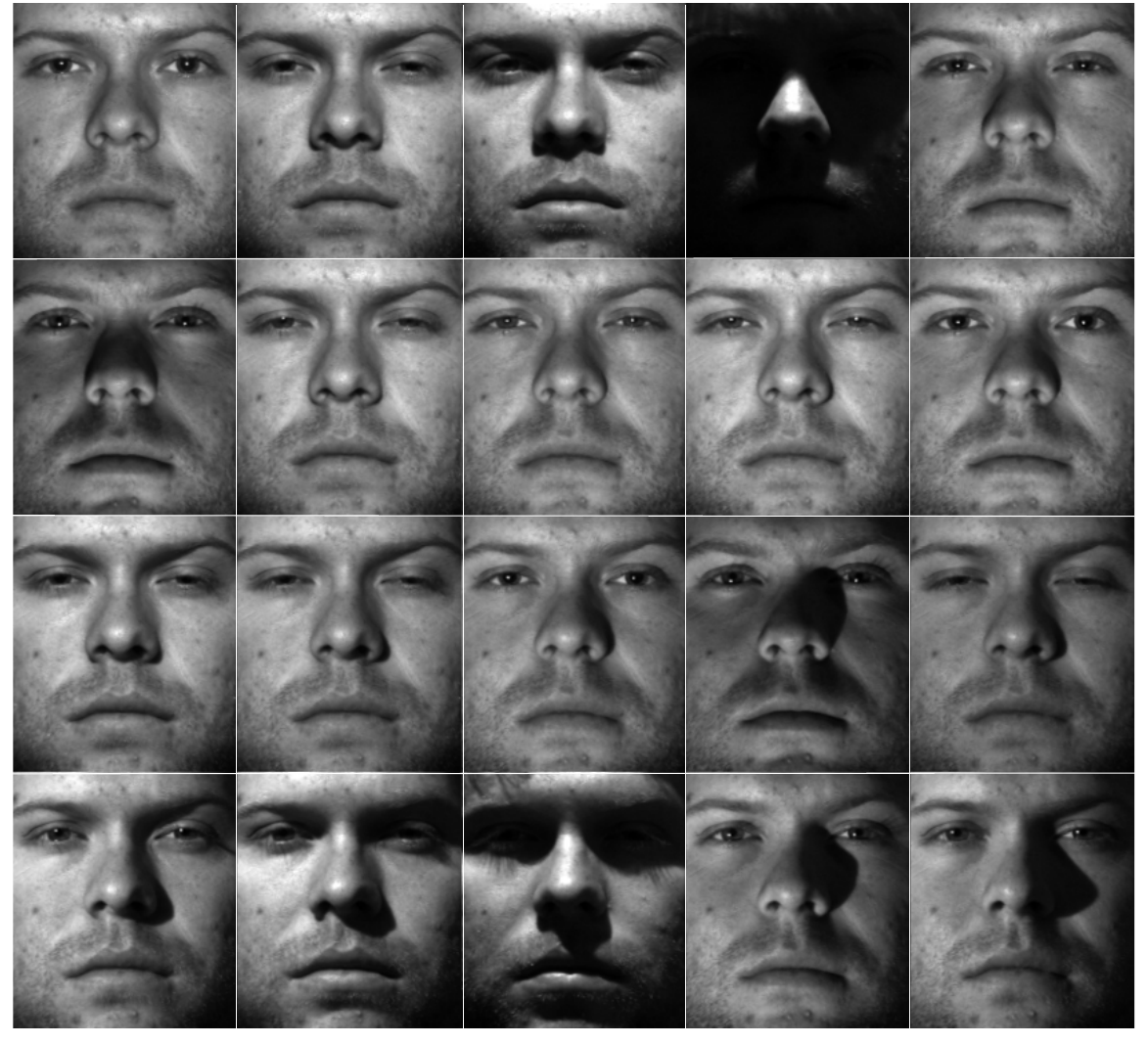}
\caption{\it Sample images for one individual. } \label{Fig5.30}
 \end{figure}

\begin{figure}
\centering
\includegraphics[width=0.8\textwidth,height=9cm]{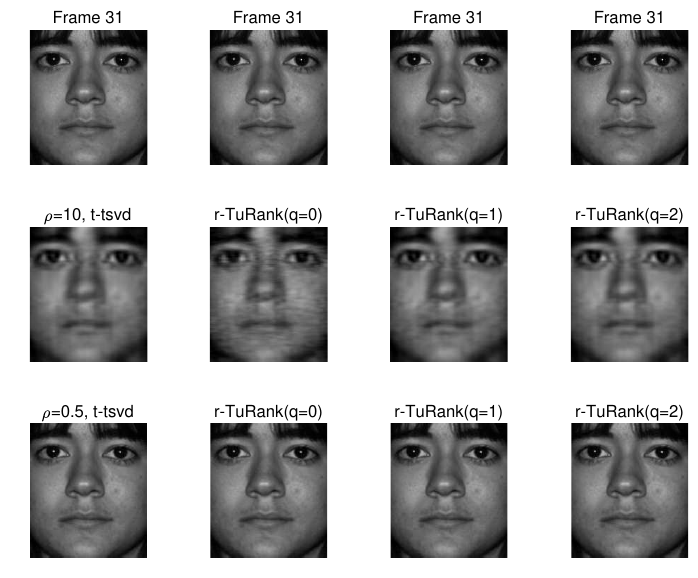}
\caption{\it The quality of compression for Frame 31 chosen from Extended Yale B dataset with $b=20, \tau={1\over n_3}\rho{\|{\cal A}\|_F}$ and $\rho=10, 0.5$.  } \label{Fig5.3}
 \end{figure}

\begin{figure}
\centering
\includegraphics[width=\textwidth,height=4.2cm]{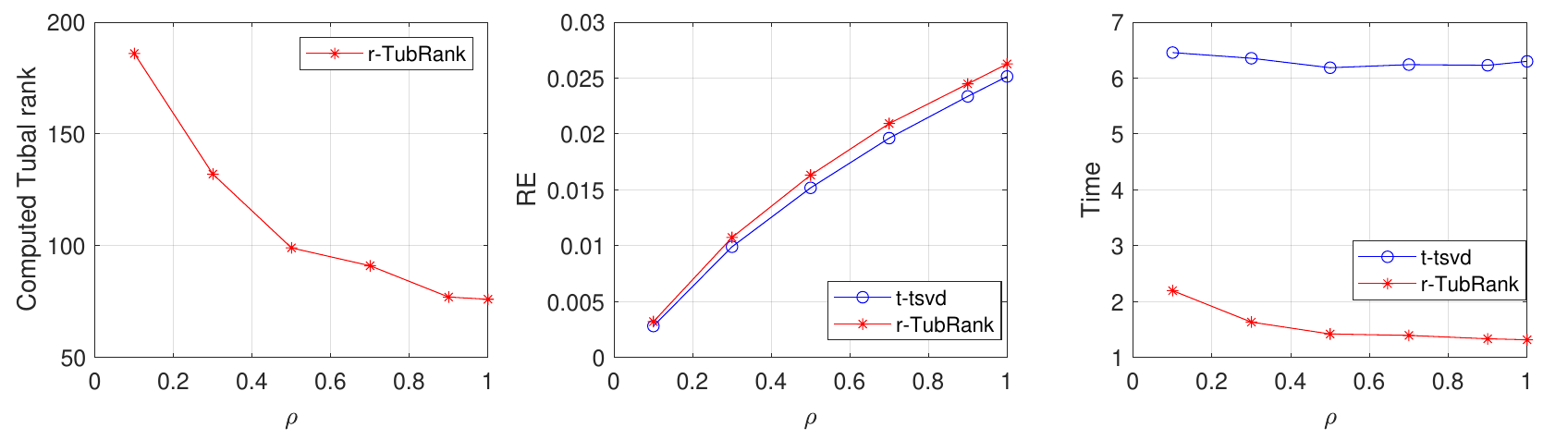}
\caption{\it Relative errors and running time comparison for computing  approximate tensor  for $b=20$ and $q=0$. } \label{Fig5.4}
 \end{figure}

\subsection{Tensor compression on the Extended Yale B data set}

We apply the  t-SVD and  {\sf r-TuRank} algorithm on the Extended Yale B data set to evaluate the accuracy of the proposed low-rank representations,
as well as its efficiency. The database consists of 2,432 images of 38 individuals, each
under 64 different lighting conditions \cite{gbk}. We used 20 images from each individual and kept
them at full resolution. Each image comprises of $192\times 168$ pixels on a grayscale range and the 760 images can be  organized into a $192\times 760\times 168$ tensor, and each lateral slice corresponds to a grayscale image of one individual. Fig. \ref{Fig5.30} depicts the sample images for one individual.

Owing to the spatial correlation within lateral slices, the tensor is expected to exhibit a low tubal rank. We set \(b = 20\) and \(\tau=\frac{1}{n_3}\rho \|{\cal A}\|_F\), where \(\rho\) varies from 0.1 to 1, characterizing specific percentages of the average energy of the frontal slices. By varying \(q = 0, 1, 2\) and employing the {\sf t-tsvd} and {\sf r-TuRank} algorithms to compute the approximate tensor, we present the results in Fig. \ref{Fig5.3}. It can be concluded that: a) in terms of \(\tau\), both {\sf t-tsvd} and {\sf r-TuRank} generate approximate images with a visual effect closer to that of the original image; b) regarding \(q\), there is little discrepancy between the approximate images obtained by {\sf t-tsvd} and {\sf r-TuRank}. This is probably because the compression problem is a low tubal rank approximation problem, and the singular values of \(\widehat {\cal A}^{(i)}\) decay rapidly. In such cases, \(q = 0\) is sufficient to produce an approximate image comparable to that generated by the {\sf t-tsvd} algorithm.

Varying $\rho$ as 0.1, 0.3, $\ldots$, 0.9, 1, we obtain the results presented in Fig. \ref{Fig5.4}. It is observed that the larger the parameter $\tau$, the lower the computed tubal rank. Across all cases of $\tau$, the accuracy of {\sf r-TuRank} is comparable to that of {\sf t-tsvd}, while it exhibits higher efficiency.

\subsection{Background modeling on CDnet-2014 data set}

In this subsection, we consider the background modeling problem
which aims to separate the foreground objects from the background.
The frames of the background of video surveillance  are highly correlated and thus can be modeled as a low tubal rank tensor. The moving foreground
objects occupy only a fraction of image pixels and thus can
be treated as sparse errors.

The background modelling problems needs the tensor robust principal component analysis (TRPCA) \cite{lcl,lcl2,zea} to recover the corrupted low-rank tensors  via  convex optimization  problem \cite{lcl}
\[
\min\limits_{{\cal L, E }}\|{\cal L}\|_\circledast+\lambda \|{\cal E}\|_1,\quad \mbox{subject to}\quad {\cal A}={\cal L}+{\cal E}\in {\mathbb R}^{n_1\times n_2\times n_3}
\]
where $\lambda>0$ is a tuning parameter,  the tensor nuclear norm (TNN) of ${\cal X}$ is defined by $
\|{\cal X}\|_\circledast={1\over n_3}\sum\limits_{k=1}^{n_3}\|\widehat {\cal X}^{(k)}\|_*,$
with $\|\widehat {\cal X}^{(k)}\|_*$ denoting the nuclear norm of its argument, i.e., the sum of singular values of $\widehat {\cal X}^{(k)}$.  The tensor  ${\ell}_1$-norm is the sum of the absolute values of all entries of ${\cal E}$.

   The TRPCA can be solved by  ADMM as the following iteration  \cite{lcl,lcl2}:

1) ${\cal L}_{k+1}=\mathop{\arg\min}\limits_{{\cal L}\in {\mathbb R}^{n_1\times n_2\times n_3}} \|{\cal L}\|_\circledast +{\mu_k\over 2}\|{\cal L}+{\cal E}_k-{\cal A}+{1\over \mu_k}{\cal Y}_k\|_F^2$;

2) ${\cal E}_{k+1}=\mathop{\arg\min}\limits_{{\cal E}\in {\mathbb R}^{n_1\times n_2\times n_3}} \|{\cal E}\|_1 +{\mu_k\over 2}\|{\cal L}_{k+1}+{\cal E}-{\cal A}+{1\over \mu_k}{\cal Y}_k\|_F^2$;

3) ${\cal Y}_{k+1}={\cal Y}_k+\mu_k ({\cal L}_{k+1}+{\cal E}_{k+1}-{\cal A})$;

4) $\mu_{k+1}=\min\{\rho\mu_k, \mu_{\rm max}\}$,\\
where  $\rho\ge 1$ and ${\cal L}_0,{\cal E}_0$ and ${\cal Y}_0$ can be taken as  zero tensors.

The key step  in ADMM for solving TRPCA   requires to solve
\begin{equation}
\min\limits_{{\cal X}\in {\mathbb R}^{n_1\times n_2\times n_3}} \tau \|{\cal X}\|_\circledast+{1\over 2}\|{\cal X}-{\cal Y}\|_F^2,\quad \tau>0,\label{eq4.6}
\end{equation}
whose minimizer ${\cal X}_\diamond$ can be obtained by the tensor Singular Value Thresholding (t-SVT) operator \cite{lcl}
 ${\mathbb S}_{\tau}(\cdot)$ on the t-SVD  of ${\cal Y}$: ${\cal Y}={\cal U}*{\cal S}*{\cal V}^T$,  i.e.,
\begin{equation}
{\cal X}_\diamond={\mathbb S}_\tau({\cal Y}):={\cal U}*{\cal S}_\tau*{\cal V}^T,\quad \mbox{or}\quad \widehat {X}_\diamond=\widehat { U}\widehat { S}_\tau \widehat { V}^T,\notag
\end{equation}
where  ${\cal S}_\tau={\sf ifft}_3({\mathscr S}_\tau(\widehat S))$  with the componentwise soft-thresholding operator ${\mathscr S}_\tau(x)={\rm sign}(x)*\max(|x|-\tau, 0)$.
To avoid the large computational cost, we consider  our {\sf r-TuRank} algorithm to compute
$$
\widehat Y^{(i)}\approx Q^{(i)}Q^{(i)^H}\widehat Y^{(i)}=Q^{(i)}L^{(i)}P^{(i)^H},
$$
where $P^{(i)}L^{(i)^H}$ is the thin QR factorization of $\widehat Y^{(i)^H}Q^{(i)}$. By Theorem \ref{thm:2.2}, we know that
$L^{(i)}$ tends to a diagonal matrix if $\widehat Y^{(i)}$ has singular values with fast decaying rate.  We apply the operator
${\mathscr S}_\tau(\cdot)$  on $L^{(i)}$ to obtain $L^{(i)}_\tau$ and $\widehat Y^{(i)}_\diamond := Q^{(i)}L^{(i)}_\tau P^{(i)^H}$.
The tensor  ${\cal X}_\diamond$  is approximated by    ${\cal X}_\diamond'={\sf ifft}_3(\widehat {\cal Y}_\diamond)$.

We choose 120 frames ($f$) of two color videos {\it pedestrians} ($360\times 240$), {\it streetLight} ($320\times 240$)   in CDnet-2014 \cite{wjp}
and resize the $h\times w$ video  to $150\times 100$, $160\times 120$, respectively. The video is reshaped into an $(hw)\times f\times 3$ tensor whose entries range from 0 to 1.
We solve the corresponding TRPCA model via {\sf r-TuRank} and {\sf t-tsvd}, respectively, in which the parameters $\lambda, \mu_0$ in ADMM  are set to   $\lambda=(3hw)^{-1/2}$, and  $\mu_0=\rho_0\lambda$ with $\rho_0=0.01, 0.05, 0.1$, respectively.  In  the $k$-th step of ADMM iterations, we compute ${\cal L}_{k+1}$ and ${\cal E}_{k+1}$ via \eqref{eq4.6} with
the threshold  $\tau=\mu_{\rm max}^{-1}=\mu_0^{-1}$. The iteration of ADMM stops if the relative error
\[
{\rm RE}=\|{\cal L}_{k+1}+{\cal E}_{k+1}-{\cal X}\|_F/\|{\cal X}\|_F< 10^{-6}.
\]

\begin{table}[]
    \center
    \caption{Comparison of  {\sf r-TuRank} and {\sf t-tsvd}-based ADMM for background modeling on different  videos. }\label{tab5.1}
    \resizebox{\textwidth}{!}{
\begin{tabular}{|lllllll|l|llllll|}
\hline
\multicolumn{7}{|c|}{\it pedestrians}                                                                                                                                                                                                             & \multirow{23}{*}{} & \multicolumn{6}{c|}{\it streetLight}                                                                                                                                                                \\ \cline{1-7} \cline{9-14}
\multicolumn{1}{|c|}{$\rho_0$} & \multicolumn{3}{c|}{methods}                                                                                         & \multicolumn{1}{c|}{time}     & \multicolumn{1}{c|}{iters} & RE    &                    & \multicolumn{3}{c|}{methods}                                                                                         & \multicolumn{1}{c|}{time}   & \multicolumn{1}{c|}{iters} & RE    \\ \cline{1-7} \cline{9-14}
\multicolumn{1}{|l|}{\multirow{7}{*}{$0.01$}} & \multicolumn{3}{l|}{\sf t-tsvd}                                                                                            & \multicolumn{1}{l|}{131.1} & \multicolumn{1}{l|}{84}   & 9.3e-7 &                    & \multicolumn{3}{l|}{\sf t-tsvd}                                                                                            & \multicolumn{1}{l|}{206.0} & \multicolumn{1}{l|}{89}   & 9.9e-7 \\ \cline{2-7} \cline{9-14}
\multicolumn{1}{|l|}{}                      & \multicolumn{1}{l|}{\multirow{6}{*}{\sf r-TuRank}} & \multicolumn{1}{l|}{\multirow{2}{*}{$q=0$}} & \multicolumn{1}{l|}{$b=5$}  & \multicolumn{1}{l|}{70.9} & \multicolumn{1}{l|}{116}   & 9.3e-7 &                    & \multicolumn{1}{l|}{\multirow{6}{*}{\sf r-TuRank}} & \multicolumn{1}{l|}{\multirow{2}{*}{$q=0$}} & \multicolumn{1}{l|}{$b=5$}  & \multicolumn{1}{l|}{95.2} & \multicolumn{1}{l|}{106}   & 9.8e-7 \\ \cline{4-7} \cline{11-14}
\multicolumn{1}{|l|}{}                      & \multicolumn{1}{l|}{}                        & \multicolumn{1}{l|}{}                     & \multicolumn{1}{l|}{$b=15$} & \multicolumn{1}{l|}{81.8} & \multicolumn{1}{l|}{117}   & 8.7e-7 &                    & \multicolumn{1}{l|}{}                        & \multicolumn{1}{l|}{}                     & \multicolumn{1}{l|}{$b=15$} & \multicolumn{1}{l|}{113.3} & \multicolumn{1}{l|}{108}     & 8.9e-7
\\ \cline{3-7} \cline{10-14}
\multicolumn{1}{|l|}{}                      & \multicolumn{1}{l|}{}                        & \multicolumn{1}{l|}{\multirow{2}{*}{$q=1$}} & \multicolumn{1}{l|}{$b=5$}  & \multicolumn{1}{l|}{65.8} & \multicolumn{1}{l|}{101}   & 8.7e-7 &                    & \multicolumn{1}{l|}{}                        & \multicolumn{1}{l|}{\multirow{2}{*}{$q=1$}} & \multicolumn{1}{l|}{$b=5$}  & \multicolumn{1}{l|}{101.7} & \multicolumn{1}{l|}{99}   & 9.9e-7
\\ \cline{4-7} \cline{11-14}
\multicolumn{1}{|l|}{}                      & \multicolumn{1}{l|}{}                        & \multicolumn{1}{l|}{}                     & \multicolumn{1}{l|}{$b=15$} & \multicolumn{1}{l|}{85.0} & \multicolumn{1}{l|}{107}   & 9.1e-7 &                    & \multicolumn{1}{l|}{}                        & \multicolumn{1}{l|}{}                     & \multicolumn{1}{l|}{$b=15$} & \multicolumn{1}{l|}{123.9} & \multicolumn{1}{l|}{107}       & 9.8e-7
\\ \cline{3-7} \cline{10-14}
\multicolumn{1}{|l|}{}                      & \multicolumn{1}{l|}{}                        & \multicolumn{1}{l|}{\multirow{2}{*}{$q=2$}} & \multicolumn{1}{l|}{$b=5$}  & \multicolumn{1}{l|}{66.2} & \multicolumn{1}{l|}{98}   & 8.6e-7 &                    & \multicolumn{1}{l|}{}                        & \multicolumn{1}{l|}{\multirow{2}{*}{$q=2$}} & \multicolumn{1}{l|}{$b=5$}  & \multicolumn{1}{l|}{108.1} & \multicolumn{1}{l|}{101}    & 8.1e-7
\\ \cline{4-7} \cline{11-14}
\multicolumn{1}{|l|}{}                      & \multicolumn{1}{l|}{}                        & \multicolumn{1}{l|}{}                     & \multicolumn{1}{l|}{$b=15$} & \multicolumn{1}{l|}{93.1} & \multicolumn{1}{l|}{106}   & 8.9e-7 &                    & \multicolumn{1}{l|}{}                        & \multicolumn{1}{l|}{}                     & \multicolumn{1}{l|}{$b=15$} & \multicolumn{1}{l|}{138.5} & \multicolumn{1}{l|}{107}    & 9.8e-7 \\ \cline{1-7} \cline{9-14}
\multicolumn{1}{|l|}{\multirow{7}{*}{$0.05$}} & \multicolumn{3}{l|}{\sf t-tsvd}                                                                                            & \multicolumn{1}{l|}{103.9} & \multicolumn{1}{l|}{67}   & 9.4e-7 &                    & \multicolumn{1}{l|}{\sf t-tsvd}                    & \multicolumn{1}{l|}{}                     & \multicolumn{1}{l|}{}     & \multicolumn{1}{l|}{167.7} & \multicolumn{1}{l|}{73}   & 9.0e-7
 \\ \cline{2-7} \cline{9-14}
\multicolumn{1}{|l|}{}                      & \multicolumn{1}{l|}{\multirow{6}{*}{\sf r-TuRank}} & \multicolumn{1}{l|}{\multirow{2}{*}{$q=0$}} & \multicolumn{1}{l|}{$b=5$}  & \multicolumn{1}{l|}{60.0} & \multicolumn{1}{l|}{99}   & 9.3e-7 &                    & \multicolumn{1}{l|}{\multirow{6}{*}{\sf r-TuRank}} & \multicolumn{1}{l|}{\multirow{2}{*}{$q=0$}} & \multicolumn{1}{l|}{$b=5$}  & \multicolumn{1}{l|}{80.7} & \multicolumn{1}{l|}{91}   & 7.6e-7 \\ \cline{4-7} \cline{11-14}
\multicolumn{1}{|l|}{}                      & \multicolumn{1}{l|}{}                        & \multicolumn{1}{l|}{}                     & \multicolumn{1}{l|}{$b=15$} & \multicolumn{1}{l|}{70.2} & \multicolumn{1}{l|}{100}   & 9.5e-7 &                    & \multicolumn{1}{l|}{}                        & \multicolumn{1}{l|}{}                     & \multicolumn{1}{l|}{$b=15$} & \multicolumn{1}{l|}{97.9} & \multicolumn{1}{l|}{91}   & 9.1e-7
\\ \cline{3-7} \cline{10-14}
\multicolumn{1}{|l|}{}                      & \multicolumn{1}{l|}{}                        & \multicolumn{1}{l|}{\multirow{2}{*}{$q=1$}} & \multicolumn{1}{l|}{$b=5$}  & \multicolumn{1}{l|}{53.9} & \multicolumn{1}{l|}{84}   & 9.2e-7 &                    & \multicolumn{1}{l|}{}                        & \multicolumn{1}{l|}{\multirow{2}{*}{$q=1$}} & \multicolumn{1}{l|}{$b=5$}  & \multicolumn{1}{l|}{77.2} & \multicolumn{1}{l|}{83}   & 9.9e-7 \\ \cline{4-7} \cline{11-14}
\multicolumn{1}{|l|}{}                      & \multicolumn{1}{l|}{}                        & \multicolumn{1}{l|}{}                     & \multicolumn{1}{l|}{$b=15$} & \multicolumn{1}{l|}{71.4} & \multicolumn{1}{l|}{90}   & 9.6e-7 &                    & \multicolumn{1}{l|}{}                        & \multicolumn{1}{l|}{}                     & \multicolumn{1}{l|}{$b=15$} & \multicolumn{1}{l|}{109.1} & \multicolumn{1}{l|}{91}   & 9.8e-7
 \\ \cline{3-7} \cline{10-14}
\multicolumn{1}{|l|}{}                      & \multicolumn{1}{l|}{}                        & \multicolumn{1}{l|}{\multirow{2}{*}{$q=2$}} & \multicolumn{1}{l|}{$b=5$}  & \multicolumn{1}{l|}{54.6} & \multicolumn{1}{l|}{81}   & 9.9e-7 &                    & \multicolumn{1}{l|}{}                        & \multicolumn{1}{l|}{\multirow{2}{*}{$q=2$}} & \multicolumn{1}{l|}{$b=5$}  & \multicolumn{1}{l|}{85.7} & \multicolumn{1}{l|}{83}   & 8.7e-7 \\ \cline{4-7} \cline{11-14}
\multicolumn{1}{|l|}{}                      & \multicolumn{1}{l|}{}                        & \multicolumn{1}{l|}{}                     & \multicolumn{1}{l|}{$b=15$} & \multicolumn{1}{l|}{77.9} & \multicolumn{1}{l|}{90}   & 8.8e-7 &                    & \multicolumn{1}{l|}{}                        & \multicolumn{1}{l|}{}                     & \multicolumn{1}{l|}{$b=15$} & \multicolumn{1}{l|}{118.7} & \multicolumn{1}{l|}{91}   & 9.1e-7
 \\ \cline{1-7} \cline{9-14}
\multicolumn{1}{|l|}{\multirow{7}{*}{$0.1$}} & \multicolumn{3}{l|}{\sf t-tsvd}                                                                                            & \multicolumn{1}{l|}{94.3} & \multicolumn{1}{l|}{60}   & 9.2e-7 &                    & \multicolumn{1}{l|}{\sf t-tsvd}                    & \multicolumn{1}{l|}{}                     & \multicolumn{1}{l|}{}     & \multicolumn{1}{l|}{151.9} & \multicolumn{1}{l|}{65}   & 9.8e-7 \\ \cline{2-7} \cline{9-14}
\multicolumn{1}{|l|}{}                      & \multicolumn{1}{l|}{\multirow{6}{*}{\sf r-TuRank}} & \multicolumn{1}{l|}{\multirow{2}{*}{$q=0$}} & \multicolumn{1}{l|}{$b=5$}  & \multicolumn{1}{l|}{55.5} & \multicolumn{1}{l|}{92}   & 9.1e-7 &                    & \multicolumn{1}{l|}{\multirow{6}{*}{\sf r-TuRank}} & \multicolumn{1}{l|}{\multirow{2}{*}{$q=0$}} & \multicolumn{1}{l|}{$b=5$}  & \multicolumn{1}{l|}{74.1} & \multicolumn{1}{l|}{83}   & 8.2e-7 \\ \cline{4-7} \cline{11-14}
\multicolumn{1}{|l|}{}                      & \multicolumn{1}{l|}{}                        & \multicolumn{1}{l|}{}                     & \multicolumn{1}{l|}{$b=15$} & \multicolumn{1}{l|}{65.9} & \multicolumn{1}{l|}{93}   & 9.4e-7 &                    & \multicolumn{1}{l|}{}                        & \multicolumn{1}{l|}{}                     & \multicolumn{1}{l|}{$b=15$} & \multicolumn{1}{l|}{94.0} & \multicolumn{1}{l|}{84}   & 9.5e-7 \\ \cline{3-7} \cline{10-14}
\multicolumn{1}{|l|}{}                      & \multicolumn{1}{l|}{}                        & \multicolumn{1}{l|}{\multirow{2}{*}{$q=1$}} & \multicolumn{1}{l|}{$b=5$}  & \multicolumn{1}{l|}{48.6} & \multicolumn{1}{l|}{75}   & 9.9e-7 &                    & \multicolumn{1}{l|}{}                        & \multicolumn{1}{l|}{\multirow{2}{*}{$q=1$}} & \multicolumn{1}{l|}{$b=5$}  & \multicolumn{1}{l|}{76.7} & \multicolumn{1}{l|}{78}   & 9.7e-7 \\ \cline{4-7} \cline{11-14}
\multicolumn{1}{|l|}{}                      & \multicolumn{1}{l|}{}                        & \multicolumn{1}{l|}{}                     & \multicolumn{1}{l|}{$b=15$} & \multicolumn{1}{l|}{64.2} & \multicolumn{1}{l|}{82}   & 9.2e-7 &                    & \multicolumn{1}{l|}{}                        & \multicolumn{1}{l|}{}                     & \multicolumn{1}{l|}{$b=15$} & \multicolumn{1}{l|}{100.8} & \multicolumn{1}{l|}{85}   & 9.1e-7 \\ \cline{3-7} \cline{10-14}
\multicolumn{1}{|l|}{}                      & \multicolumn{1}{l|}{}                        & \multicolumn{1}{l|}{\multirow{2}{*}{$q=2$}} & \multicolumn{1}{l|}{$b=5$}  & \multicolumn{1}{l|}{51.5} & \multicolumn{1}{l|}{74}   & 8.8e-7 &                    & \multicolumn{1}{l|}{}                        & \multicolumn{1}{l|}{\multirow{2}{*}{$q=2$}} & \multicolumn{1}{l|}{$b=5$}  & \multicolumn{1}{l|}{77.3} & \multicolumn{1}{l|}{76}   & 9.8e-7 \\ \cline{4-7} \cline{11-14}
\multicolumn{1}{|l|}{}                      & \multicolumn{1}{l|}{}                        & \multicolumn{1}{l|}{}                     & \multicolumn{1}{l|}{$b=15$} & \multicolumn{1}{l|}{71.6} & \multicolumn{1}{l|}{83}   & 8.9e-7 &                    & \multicolumn{1}{l|}{}                        & \multicolumn{1}{l|}{}                     & \multicolumn{1}{l|}{$b=15$} & \multicolumn{1}{l|}{107.6} & \multicolumn{1}{l|}{83}   & 9.2e-7 \\ \hline
\end{tabular}}
\end{table}

\begin{figure}
\centering 
\includegraphics[width=0.9\textwidth,height=9cm]{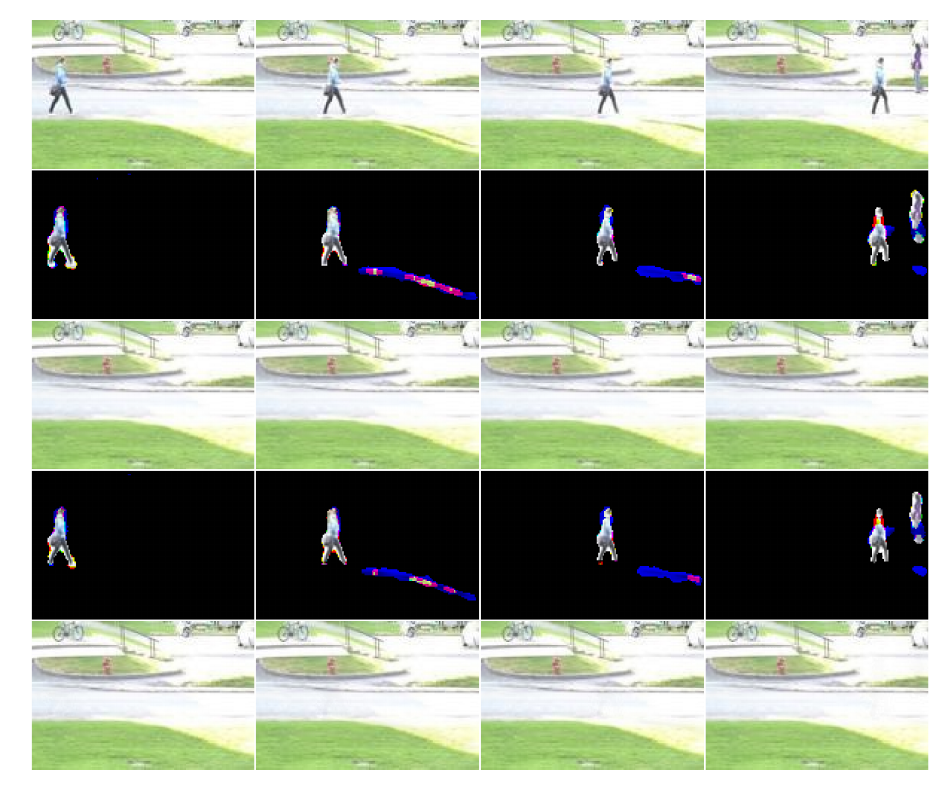}
\caption{\it   Background and  walking lady for {\sf pedestrians}. Herein, Rows 2 and 3 correspond to the {\sf r-TuRank}-based ADMM algorithm (with parameters b=5 and q=0), while Rows 4 and 5 correspond to the {\sf t-tsvd}-based algorithm. In the ADMM algorithm for both methods, the parameter $\mu_0=0.08$.}\label{fig:5.6}
 \end{figure}

 \begin{figure}[h]
\centering
\includegraphics[width=\textwidth]{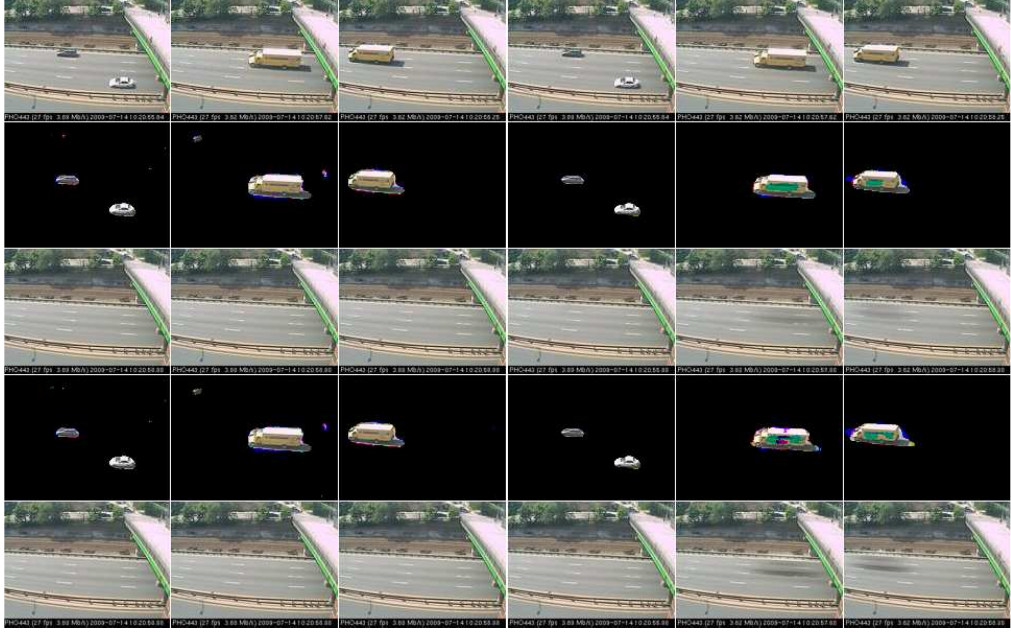}
\caption{\it  Background and moving cars  for {\sf streetLight}.  Columns 1 to 3 correspond to {\sf r-TuRank}-based ADMM (with parameters $b=5$ and $q=0$), while Columns 4 to 6 correspond to the {\sf t-tsvd}-based algorithm. Rows 2 and 3 correspond to $\mu_0=0.01\lambda$ and Rows 4 and 5 are  for  $\mu_0=0.08\lambda$. }\label{fig:5.7}
 \end{figure}

For different $\mu_0$, $b$ and $q$,   numerical results in  Table \ref{tab5.1} show that the {\sf r-TuRank}-based method outperforms the one based on {\sf t-tsvd} in terms of the running   time, although the former method needs more iterations.  For {\sf r-TuRank}, it should also be noted that for the fixed parameter $b$ and $\mu_0$, the increase of $q$ generally needs more running time but computes a more accurate approximation to {\sf t-tsvd}, and hence reduce the iterations of ADMM.

In Fig. \ref{fig:5.6} and Fig. \ref{fig:5.7}, we display the results for separating backgrounds from the foregrounds with a potentially poor parameter $q=0$. In Fig. \ref{fig:5.6}, a lady walking in the sun—along with her shadow—can be clearly detected by both methods.
In the video {\sf streetLight}, moving cars and vans are also separable by both methods; their separation results are visually comparable overall, though subtle differences emerge in details. In Fig. \ref{fig:5.7}, the foreground separated using {\sf r-TuRank} (see column 2) contains small blurry patches, while clearer features of the van's windows remain discernible. Additionally, in the restored background images via {\sf t-svd}, several patches of black shadows are visible on the road surface (see columns 5–6), whereas  these shadows are absent from the background images of columns 2–3.

\section{Conclusion}
In this paper, we present an adaptive randomized tubal rank-revealing algorithm of the data tensor ${\cal A}$ within the given threshold.
The algorithm extracts approximate basis for the range of the Fourier-domain tensor, and constructs an optimal approximation to the given tensor within the restricted range.
It avoids the large computational cost arising in the truncation of the full t-SVD.
Theoretical analysis shows
that the proposed algorithm can compute low tubal-rank approximation up to some
constants depending on the dimension of the data and Fourier-domain singular value
gap from optimal. Experimental results show its effectiveness and efficiency
in  image processing and background modeling problems.

\vskip 0.5cm

\noindent{\bf\small Data Availability} The data sets generated during and/or analyzed during the current study are
available from the corresponding author on reasonable request.

\section*{Declarations}
{\bf\small Conflict of Interest Statement} The authors declare no competing interests.

\noindent{\bf \small Ethical Approval}  Not applicable.

\bibliographystyle{unsrt}

\end{document}